\documentclass{amsart}
\usepackage{a4}

\usepackage{hyperref}
\usepackage[italian, english]{babel}

\usepackage{MnSymbol} % \per avere upmapsto e \downmapsto
\usepackage{tikz-cd} % per i diagrammi commutativi

\theoremstyle{definition}
\newtheorem{definition}{Definition}[section]
\newtheorem{theorem}[definition]{Theorem}
\newtheorem{lemma}[definition]{Lemma}
\newtheorem{proposition}[definition]{Proposition}
\newtheorem{corollary}[definition]{Corollary}
\newtheorem{remark}[definition]{Remark}
\newtheorem{code}[definition]{Code}
\newtheorem*{castelnuovo}{Castelnuovo's conjecture}

\newtheorem*{classical}{The classical Enriques-Fano threefold}
\newtheorem*{bayle13}{The Enriques-Fano threefold of genus $\mathbf{13}$ found by Bayle and Sano}
\newtheorem*{pro13}{The Prokhorov-Enriques-Fano threefold of genus $\mathbf{13}$}
\newtheorem*{pro17}{The Prokhorov-Enriques-Fano threefold of genus $\mathbf{17}$}

\usepackage{xcolor} % per scrivere colorato con {\color{nome colore}testo}
\usepackage[font={small,it}]{caption}
\usepackage{enumerate} % per numerare coi numeri romani con \begin{enumerate}[(I)]

\begin{document}

\title[On Enriques-Fano threefolds and a conjecture of Castelnuovo]{On Enriques-Fano threefolds and a conjecture of Castelnuovo}
\author[V. Martello]{Vincenzo Martello}
\address{Dipartimento di Matematica, Universit\`{a} della Calabria, Arcavacata di Rende (CS)}
\email{vincenzomartello93@gmail.com}
%\date{\today}

%\thanks{}
%\subjclass{}
%\keywords{enriques fano threefolds, castelnuovo, conjecture, elliptic ruled surface.}
%\dedicatory{}
%\commby{}

\begin{abstract}
Let $W\subset \mathbb{P}^{13}$ be the image of the rational map defined by the linear system of the sextic surfaces of $\mathbb{P}^3$ having double points along the edges of a tetrahedron. Let $\mathcal{L}$ be the linear system of the hyperplane sections of $W$. It is known that a general $S\in \mathcal{L}$ is an Enriques surface. The aim of this paper is to study the sublinear system $\mathcal{L}_{\bullet}\subset \mathcal{L}$ of the hyperplane sections of $W$ having a triple point at a general point $w \in W$. We will show that a general element of $\mathcal{L}_{\bullet}$ 
%this linear system 
is birational to an elliptic ruled surface and that the image of $W$ via the rational map defined by $\mathcal{L}_{\bullet}$ is a cubic Del Pezzo surface $\Delta\subset \mathbb{P}^3$ with $4$ nodes. Interestingly, this fact appears to be related to a conjecture of Castelnuovo.
\end{abstract}

\maketitle
%\tableofcontents

\section{Introduction}

Let $\mathcal{L}$ be an $r$-dimensional linear system of surfaces of $\mathbb{P}^3$ such that the desingularization $\widetilde{S}$ of the general element $S\in \mathcal{L}$ has zero geometric genus $p_g(\widetilde{S})=0$ and zero arithmetic genus $p_a(\widetilde{S})=0$. 
It follows that $\widetilde{S}$ is \textit{regular}, i.e. it has zero irregularity $q(\widetilde{S})=0$. We also recall that a variety is said to be \textit{irregular} if it has positive irregularity; moreover, we will say that a singular variety $X$ has \textit{regular} (respectively, \textit{irregular}) \textit{desingularization} if, for each resolution of singularities $f : \widetilde{X} \to X$, we have $q(\widetilde{X})=0$ (respectively, $q(\widetilde{X})>0$).
%Anyway, 
%one may ask: 
What happens if we force the surfaces of $\mathcal{L}$ to have a triple point at a general point of $\mathbb{P}^3$? Castelnuovo \textit{conjectured} in \cite[pp. 187-188]{Ca37} that one gets an $(r-10)$-dimensional sublinear system $\mathcal{L}_{\bullet}\subset \mathcal{L}$ whose general element satisfies one of the following three properties: 
\begin{itemize}
\item[(A)] it is an irreducible surface whose desingularization is an \textit{irregular} surface with zero geometric genus and with arithmetic genus equal to $-1$; 
\item[(B)] it is \textit{reducible} in two rational surfaces intersecting along a rational curve; 
\item[(C)] it has the \textit{same genera} as a general element of $\mathcal{L}$. 
\end{itemize}
Castelnuovo also observed that the first property is impossible if $r>19$. Moreover, he stated that, if $13\le r \le 19$ and if the first property occurs, the image of $\mathbb{P}^3$ via the rational map
defined by $\mathcal{L}_{\bullet}$ is a Del Pezzo surface in $\mathbb{P}^{r-10}$ of degree $3 \le r-10 \le 9$. We will apply the arguments of Castelnuovo to (rational) regular smooth irreducible threefolds
(see \S~\ref{sec: conjecture}). A natural adaptation to (rational) 
threefolds with isolated singularities and with regular desingularization
follows, as we will see in \S~\ref{sec: conjecture singular}. 

An example of such threefolds is the \textit{classical Enriques-Fano threefold}, which was found by Fano in \cite[\S 8]{Fa38} and which is the image of the $13$-dimensional linear system of the sextic surfaces of $\mathbb{P}^3$ having double points along the six edges of a tetrahedron. 
Let $W\subset \mathbb{P}^{13}$ be the classical Enriques-Fano threefold and let $\mathcal{L}$ be the $13$-dimensional linear system of its hyperplane sections. It is known that $W$ is rational, that $W$ has only eight singular points, and that a general $S\in \mathcal{L}$ is an Enriques surface, which is a smooth irreducible surface with $p_g(S) = p_a(S) = 0$.
We will study the sublinear system $\mathcal{L}_{\bullet}\subset \mathcal{L}$ of the hyperplane sections of the classical Enriques-Fano threefold $W$ having a triple point at a general point $w \in W$ (see \S~\ref{sec: classical EF3-fold}). We will find that the general element of this sublinear system satisfies the property (A) conjectured by Castelnuovo. In particular, we will prove the following result.
% and that it is birational to an elliptic ruled surface (see Theorem~\ref{thm:MAINthmCONJ}).

\begin{theorem}\label{thm:MAINthmCONJ}
Let $(W,\mathcal{L})$ be the classical Enriques-Fano threefold. Let $\mathcal{L}_{\bullet}\subset \mathcal{L}$ be the sublinear system of the hyperplane sections of $W$ having a triple point at a general point $w \in W$. Then
\begin{itemize}
\item[(i)] a general $S_{\bullet}\in \mathcal{L}_{\bullet}$ is birational to an elliptic ruled surface;
\item[(ii)] the image of $W$ via the rational map defined by $\mathcal{L}_{\bullet}$ is a Cayley cubic surface, 
that is a cubic Del Pezzo surface in $\mathbb{P}^{3=13-10}$ with four singular points.  
\end{itemize}
\end{theorem}

%Furthemore we will prove that the image of the classical Enriques-Fano threefold $W$ via the rational map defined by $\mathcal{L}_{\bullet}$ is a Cayley cubic surface, which is a cubic Del Pezzo surface in $\mathbb{P}^{3=13-10}$ with four singular points. 

More generally, an Enriques-Fano threefold is a normal threefold $W$ endowed with a complete linear system $\mathcal{L}$ of ample Cartier divisors such that the general element $S\in \mathcal{L}$ is an Enriques surface and such that $W$ is not a generalized cone over $S$. Every Enriques-Fano threefold is known to be singular with isolated singular points (see \cite{Ch96} and \cite[Lemma 3.2]{CoMu85}) and to have regular desingularization (see \cite[Lemma 4.1]{CDGK20}); nevertheless, the classification of these objects still remains an open question. 

In \S~\ref{sec:EFGENUSgreater13} we will show that the Enriques-Fano threefold found by Bayle in \cite[\S 6.3.2]{Ba94} (and also by Sano in \cite[Theorem 1.1 No.14]{Sa95}) coincides with the classical Enriques-Fano threefold: we will do this thanks to a computational analysis with the software Macaulay2. 
Finally, let $(W,\mathcal{L})$ be the Enriques-Fano threefold found by Prokhorov in \cite[Proposition 3.2]{Pro07}, which we will refer to as \textit{Prokhorov-Enriques-Fano threefold of genus 17}: we will see that also by imposing a triple point at the general element of $\mathcal{L}$, one obtains a surface whose desingularization has $q=1$, $p_g=0$ and $p_a=-1$. Indeed, in \S~\ref{sec: generalization}, we will prove the following result. 

\begin{corollary}\label{cor:castelnuovop17}
Let $(W,\mathcal{L})$ be the Prokhorov-Enriques-Fano threefold of genus $17$. Let $\mathcal{L}_{\bullet}\subset \mathcal{L}$ be the sublinear system given by the elements of $\mathcal{L}$ with a triple point at a general point $w\in W$. Then a general element $S_{\bullet}\in \mathcal{L}_{\bullet}$ is birational to an elliptic ruled surface.
\end{corollary} 

In Appendix~\ref{app:code} we will collect the input codes used in Macaulay2.
We will work over the field $\mathbb{C}$ of the complex numbers, but, for the computational analysis, we will work over a finite field: we will choose $\mathbb{F}_n := \mathbb{Z}/n\mathbb{Z}$ with $n=10000019$.

\subsection*{Acknowledgment}
The results of this paper are contained in my PhD-thesis. I would like to thank my main advisors C. Ciliberto and C. Galati and my co-advisor A.L. Knutsen for our stimulating conversations and for providing me very useful suggestions.
I would also like to acknowledge PhD-funding from the Department of Mathematics and Computer Science of the University of Calabria. 
%and funding from Research project ”Families of curves: their moduli and their related varieties” (CUP E81—18000100005, P.I. Flaminio Flamini) in the framework of Mission Sustainability 2017 - Tor Vergata University of Rome.

%\section{Terminology}
%
%Let us denote by $K_X$ the \textit{canonical divisor} of a smooth projective variety $X$. We recall that the numbers $p_{g}(X) := h^{0}(X,\mathcal{O}_{X}(K_{X}))$ and $p_{a}(X) :=(-1)^{\dim X} (\chi (\mathcal{O}_{X})-1)$ are called, respectively, the \textit{geometric genus} and the \textit{arithmetic genus} of $X$.
%Furthermore the \textit{irregularity} of a projective variety $X$ is the number $q(X) := h^1(X,\mathcal{O}_{X})$ and $X$ is called \textit{regular} if $q(X) = 0$, otherwise it is said to be \textit{irregular}. 

\section{Castelnuovo's conjecture for smooth threefolds}\label{sec: conjecture}

In \cite[pp.187-188]{Ca37}, Castelnuovo proposed some ideas about certain irreducible threefolds and particular linear systems of surfaces on them. In order to explain these ideas, 
we start by talking about the
link between the irregularity of a surface contained in a threefold and the one of the threefold itself, which was studied in \cite[\S 4]{CE1906}.

\begin{proposition}\label{prop:C-E qW=qS}
% ENUNCIATO CASTELNUOVO-ENRIQUES:
%Let $W$ be a smooth threefold endowed with an $r$-dimensional linear system $\mathcal{L}$, with $r\ge 2$ and such that the elements of $\mathcal{L}$ are irregular surfaces. If the intersection curves of two general elements $S,S'\in \mathcal{L}$ are irreducible curves, then $W$ is irregular and it has the same irregularity of the surfaces in $\mathcal{L}$, i.e. $q(W) = q(S)$ for each $S\in \mathcal{L}$.
Let $W$ be a smooth irreducible threefold endowed with an $r$-dimensional linear system $\mathcal{L}$ such that $r\ge 2$ and such that the general element is 
an irreducible surface. If the divisors of $\mathcal{L}$ are big and nef, then $W$ has the same irregularity as a general surface $S\in \mathcal{L}$.
\end{proposition}
\begin{proof}
Let $S$ be a general element of $\mathcal{L}$ and let us take the 
following 
exact sequence
$$0\to \mathcal O_W(-S)\to  \mathcal O_W\to \mathcal O_S\to 0.$$
Since $S$ is a big and nef divisor, then we have that
$h^{i=1,2}(\mathcal O_W(-S) )=0$ by the Kawamata-Viehweg vanishing theorem and therefore we obtain $q(W) = q(S)$.
\end{proof}

\begin{remark}\label{rem:basecurves}
Let $W$ be a smooth irreducible threefold and let $\mathcal{L}$ be a linear system on $W$ such that $\dim \mathcal{L}\ge 2$ and such that the general element is a smooth irreducible surface. Let us suppose that $\mathcal{L}$ has \textit{base curves}. One can take appropriate blow-ups of $W$ along these curves 
%in order to obtain 
until one gets 
a birational morphism $\overline{bl} : \overline{W} \to W$ such that the strict transform $\overline{\mathcal{L}}$ of $\mathcal{L}$ has no base curves. Let $\overline{S}$ be an 
element of $ \overline{\mathcal{L}}$. It follows that $\overline{S}$ is a nef divisor. Furthemore $\overline{S}$ is the strict transform of an element $S\in\mathcal{L}$. If $S$ is a big divisor, then $\overline{S}$ is big too.
\end{remark}

Some consequences of Proposition~\ref{prop:C-E qW=qS} are stated in \cite[\S 6]{CE1906} for the $3$-dimensional projective space $W = \mathbb{P}^{3}$. We will adapt them for any regular smooth irreducible threefold. 

\begin{proposition}\label{prop:C-E Sregular}
% ENUNCIATO CASTELNUOVO-ENRIQUES:
%Let $\mathcal{L}$ be an $r$-dimensional linear system of surfaces on a regular threefold $W$, with $r\ge 2$. If the intersection curve of two general surfaces in $\mathcal{L}$ is irreducible, then the general surface in $\mathcal{L}$ is regular.
Let $W$ be a regular smooth irreducible threefold endowed with an $r$-dimensional linear system $\mathcal{L}$ such that $r\ge 3$ and such that the general element is a smooth irreducible surface. If the intersection of two general surfaces of $\mathcal{L}$, outside the base locus, is an irreducible curve, then a general element $S\in \mathcal{L}$ is a regular surface.
\end{proposition}
\begin{proof}
We may assume that the base locus of $\mathcal{L}$ is empty or at worst a finite set. Indeed, if this were not the case, the proof could proceed with the pair $(\overline{W}, \overline{\mathcal{L}})$ of Remark~\ref{rem:basecurves} instead of the pair $(W,\mathcal{L})$. This is possible since a general $\overline{S}\in \overline{\mathcal{L}}$ is a smooth surface isomorphic to a general $S\in \mathcal{S}$ and such that $q(\overline{S})=q(S)$.

Let us now fix a general $S\in\mathcal{L}$, which is a nef divisor,
%on $W$ by the Zariski-Fujita theorem. 
and
let us suppose that $S$ is an irregular surface. Let $\Delta\subseteq \mathbb{P}^{r}$ be the image of the rational map $\phi_{\mathcal{L}} : W \dashrightarrow \mathbb{P}^{r}$ defined by $\mathcal{L}$. Since $r\ge 3>1$ and $S$ is 
%an 
irreducible 
%surface 
by hypothesis, then $\dim \Delta >1$. 
Furthermore, we have that $\dim \Delta < 3$, since $S$ is not a big divisor (see Proposition~\ref{prop:C-E qW=qS}). 
So $\Delta$ is a surface and the general fibre of $\phi_{\mathcal{L}}$ is a curve. 
Let $S'$ be another general element of $\mathcal{L}$. The intersection curve $S\cap S'$ is sent by $\phi_{\mathcal{L}}$ to the intersection of two general hyperplane sections of $\Delta$, which is a set of $d:=\deg \Delta$ points of $\Delta$. We observe that $\Delta$ cannot be a plane, since $r\ge 3>2$. Hence we have that $S\cap S'$ is a reducible curve given by $d\ge 2$ fibres of $\phi_{\mathcal{L}}$. Since this is a contradiction with the hypothesis, then $S$ must be regular.
\end{proof}

%As in \cite[\S 6]{CE1906}, let us describe a consequence of Corollary~\ref{cor:C-E Sregular}.

We recall that a one-dimensional linear system on a variety is called \textit{pencil}. 
%In the following we will 
Let us
extend the use of this term.
Let $S$ be a smooth surface and let $B$ be a smooth curve of genus $b\ge 0$. A surjective rational map $f : S \dashrightarrow B$ with connected fibres is called \textit{pencil of genus} $b$ of curves on $S$. All the curves of such a pencil are linearly equivalent if and only if $b = 0$. In this case we will refer to it as \textit{rational pencil}. If $b>0$, we will talk about \textit{irrational pencil} and, in this case, $f$ is a morphism (see \cite[p.114]{BPV84}). In particular, an irrational pencil of genus one is called \textit{elliptic pencil}.
% vedi questa definizione nell'articolo Catanese-Ciliberto-MendesLopes

%\begin{definition}
%An \textit{irrational pencil} of genus $b > 0$ on a smooth surface $S$ is a morphism $f : S \to B$ with connected fibres, where $B$ is a smooth curve of genus $b$. An irrational pencil of genus $1$ is also called \textit{elliptic genus}.
%\end{definition}

\begin{definition}\label{def:congruence}
A \textit{congruence of curves} of a threefold $W$ is a two-dimensional irreducible family $\mathcal{V}$ of curves
contained in $W$ such that only one curve of the family passes through a general point of $W$.
\end{definition}

\begin{proposition}\label{prop:C-E L*}
% ENUNCIATO CASTELNUOVO-ENRIQUES
%Let $\mathcal{L}_{\bullet}$ be an $r_{\bullet}$-dimensional linear system of irregular surfaces on a regular threefold $W$, with $r_{\bullet}\ge 2$. Two general elements $S_{\bullet}, S_{\bullet}' \in \mathcal{L}_{\bullet}$ intersect each other (outside the base locus) in reducible curves, the components of which form, on a fixed $S_{\bullet}$, the fibers of an irrational pencil. Furthermore, by varying the surface $S_{\bullet}$, these component curves form in $W$ a congruence $\mathcal{V}$ of curves.
Let $W$ be a regular smooth irreducible threefold endowed with an $r_{\bullet}$-dimensional linear system $\mathcal{L}_{\bullet}$, where $r_{\bullet}\ge 3$, such that the general element is an irregular smooth irreducible surface. Then two general elements $S_{\bullet}$ and $S_{\bullet}'$ of $\mathcal{L}_{\bullet}$ intersect each other (outside the base locus) along a reducible curve. In particular, on a fixed $S_{\bullet}$, the components of such a curve are fibres of a pencil of genus $b$, where $0\le b \le q(S_{\bullet})$. Furthermore, by moving the surface $S_{\bullet}$, these component curves give a congruence $\mathcal{V}$ of curves of $W$.
\end{proposition}
\begin{proof}
We may assume that $\mathcal{L}_{\bullet}$ is base point free. Indeed, if $\mathcal{L}_{\bullet}$ had base curves, then we could take the pair $(\overline{W}, \overline{\mathcal{L}}_{\bullet})$ as in Remark~\ref{rem:basecurves}, where $\overline{\mathcal{L}}_{\bullet}$ has no base curves. If $\overline{\mathcal{L}}_{\bullet}$ still had a finite set of base points, then we could blow-up such points until we get a birational morphism $\widetilde{bl}: \widetilde{W} \to \overline{W}$ such that the strict transform $\widetilde{\mathcal{L}}_{\bullet}$ of $\overline{\mathcal{L}}_{\bullet}$ is base point free. Thus, we could continue the proof by denoting the pair $(\widetilde{W}, \widetilde{\mathcal{L}}_{\bullet})$ by $(W,\mathcal{L}_{\bullet})$: this is possible since a general surface $\widetilde{S}_{\bullet}\in\widetilde{\mathcal{L}}_{\bullet}$ is birational to a general surface $S_{\bullet}\in \mathcal{L}_{\bullet}$ and they have the same irregularity.

Therefore, since 
the divisors of $\mathcal{L}_{\bullet}$ are 
%nef but 
not big (see Proposition~\ref{prop:C-E qW=qS}), then the image of the morphism $\phi_{\bullet} : W \to \mathbb{P}^{r_{\bullet}}$ defined by $\mathcal L_{\bullet}$ is not a threefold. Moreover, since $r_{\bullet}\ge 3>1$ and the elements of $\mathcal{L}_{\bullet}$ are generically irreducible, then $\phi_{\bullet} (W)$ is not even a curve. The image of $W$ via $\phi_{\bullet}$ is thus a surface $\Delta$ and a general $S_{\bullet} \in \mathcal L_{\bullet}$ is sent via $\phi_{\bullet}$ to a curve $\Gamma$, which is a general hyperplane section of $\Delta$. 
Since $S_{\bullet}$ is smooth, the morphism $\phi_{\bullet}|_{S_{\bullet}} : S_{\bullet} \to \Gamma$ factorizes via the normalization $n: B \to \Gamma$ of $\Gamma$, i.e. there exists a morphism $\psi : S_{\bullet} \to B$ such that $\phi_{\bullet}|_{S_{\bullet}} = n \circ \psi$. Furthermore,
the fibres of $\psi : S_{\bullet} \to B$ are generically equal to the ones of $\phi_{\bullet}|_{S_{\bullet}} : S_{\bullet} \to \Gamma$.
By Proposition~\ref{prop:C-E Sregular} we have that the curves on $S_{\bullet}$ given by the intersection with other general elements of $\mathcal{L}_{\bullet}$ are reducible; in particular the components of such curves are fibres of the map $\phi_{\bullet} |_{S_{\bullet}} : S_{\bullet} \to \Gamma$.
We observe that $0\le b:=p_g(B) = p_g(\Gamma)\le q(S_{\bullet})$, since we have the injection $H^0( \Omega_{\Gamma}^{1} ) \hookrightarrow  H^0(\Omega_{S_{\bullet}}^{1})$.
Finally, by moving the surface $S_{\bullet}$, we obtain that the fibres of the morphism $\phi_{\bullet} : W \to \Delta \subset \mathbb{P}^{r_{\bullet}}$ give a two dimensional family $\mathcal{V}$ such that through a general point $w\in W$ only one curve of the family passes, that is $\phi_{\bullet}^{-1}(\phi_{\bullet} (w))$.
\end{proof}

If we take $W=\mathbb{P}^{3}$ as in \cite[\S 6]{CE1906}, or more in general a \textit{rational} smooth irreducible threefold, instead of any regular smooth irreducible threefold, we obtain an additional property. Let us see which one.

\begin{remark}\label{rem:rationality congruence if W rational}
Let $(W,\mathcal{L}_{\bullet})$ be a pair given by a threefold and a linear system satisfying the hypothesis of Proposition~\ref{prop:C-E L*}. If $W$ is rational, the congruence $\mathcal{V}$ of curves of $W$ is parametrized by a rational surface $R$. Let us explain why.
Through a general point $w\in W$ only one curve $\gamma_{w}\in \mathcal{V}$ passes (see Definition~\ref{def:congruence}). If $R$ is the surface parametrizing the curves of $\mathcal{V}$, let $r_{w}$ be the point of $R$ corresponding to the curve $\gamma_{w}$. We have a dominant rational map $W \dashrightarrow R$ such that $w \mapsto r_{w}$. Since $W$ is rational, then $R$ is unirational, and so, as consequence of the Castelnuovo Rationality criterion, $R$ is rational.
\end{remark}

\begin{castelnuovo}
Let us take a rational smooth irreducible threefold $W$ and an $r$-dimensional linear system $\mathcal{L}$ on $W$ such that a general $S \in \mathcal{L}$ is a smooth irreducible surface with zero geometric genus $p_{g}(S) = 0$ and zero arithmetic genus $p_{a}(S)=0$.
Let $\mathcal{L}_{\bullet}$ be the sublinear system of $\mathcal{L}$ given by the surfaces of $\mathcal{L}$ having a triple point at a general point $w \in W$. 
Then the linear system $\mathcal{L}_{\bullet}$ has dimension $r-10$ and one of the following conditions occurs:
\begin{enumerate}[(A)]
\item\label{castelnuovoA} a general element $S_{\bullet}\in\mathcal{L}_{\bullet}$ is an irreducible surface which has irregular desingularization $\widetilde{S}_{\bullet}$ with $q(\widetilde{S}_{\bullet}) = 1$, $p_g(\widetilde{S}_{\bullet}) = 0$ and $p_a(\widetilde{S}_{\bullet}) = -1$;
\item\label{castelnuovoB} the surfaces $S_{\bullet}\in\mathcal{L}_{\bullet}$ are reducible in the union $S_{\bullet} = F_{\bullet} \cup M_{\bullet}$ of two rational surfaces passing through the point $w$, where the surface $M_{\bullet}$ changes by moving $S_{\bullet}$, the surface $F_{\bullet}$ is fixed and $F_{\bullet} \cap M_{\bullet}$ is a rational curve; 
\item\label{castelnuovoC} the surfaces $S_{\bullet}\in\mathcal{L}_{\bullet}$ have the same genera as a general $S \in \mathcal{L}$. 
\end{enumerate}
\end{castelnuovo}

Let us suppose that case (\ref{castelnuovoA}) of Castelnuovo's conjecture occurs.
Let us consider the blow-ups necessary to obtain a birational morphism $bl: \widetilde{W} \to W$ such that the strict transform $\widetilde{S}_{\bullet}$ of $S_{\bullet}$ is a smooth irreducible surface 
%and such that it moves 
moving
in an $r$-dimensional base point free linear system, given by the strict transform $\widetilde{\mathcal{L}}_{\bullet}$ of $\mathcal{L}_{\bullet}$.
If $r\ge 13$, then $r_{\bullet} := \dim \mathcal{L}_{\bullet} = r-10 \ge 3$ and we can apply Proposition~\ref{prop:C-E L*} to the pair $(\widetilde{W}, \widetilde{\mathcal{L}}_{\bullet})$. Thus, the intersection of two general surfaces of $\widetilde{\mathcal{L}}_{\bullet}$ is the union of some elements of a congruence 
of curves of $\widetilde{W}$. 
These curves are fibres of a pencil of genus $b$ on a general surface $\widetilde{S}_{\bullet}\in \widetilde{\mathcal{L}}_{\bullet}$, where $0\le b \le q(\widetilde{S}_{\bullet})=1$. 
In particular, if $\widetilde{\phi}_{\bullet} : \widetilde{W} \to \mathbb{P}^{r_{\bullet}}$ is the morphism defined by $\widetilde{\mathcal{L}}_{\bullet}$, 
we have that $b:=p_g(\Gamma)$ where $\Gamma := \widetilde{\phi}_{\bullet} (S_{\bullet})$.
Furthermore,
$\Delta := \widetilde{\phi}_{\bullet} (\widetilde{W})$
is a rational surface of $\mathbb{P}^{r_{\bullet}}$ with general hyperplane section $\Gamma$ (see Remark~\ref{rem:rationality congruence if W rational}). 

\begin{remark}\label{rem:caseEllipticGamma}
%If $\Gamma$ is a rational curve, then $\Delta$ could be a rational normal scroll or a Veronese surface or its projections \textbf{referenza?}.
If case (\ref{castelnuovoA}) of Castelnuovo's conjecture occurs, if $r\ge 13$ and if $\Gamma$ is an elliptic curve, then $\Delta \subset \mathbb{P}^{r_{\bullet}}$ is a Del Pezzo surface (see \cite[VI, Exercise (1)]{Beau78}). In this case $\Delta \subset \mathbb{P}^{r_{\bullet}}$ is represented on the projective plane $\mathbb{P}^2$ by a linear system $\mathcal{D}$ of elliptic curves with $\dim \mathcal{D}\le 9$. Since the linear system $\mathcal{L}_{\bullet}$ is in birational correspondence with the linear system $\mathcal{D}$, we have $\dim \mathcal{D} = \dim \mathcal{L}_{\bullet} = r-10\le 9$, which implies $r \le 19$. 
\end{remark}

\section{Castelnuovo's conjecture for singular threefolds}\label{sec: conjecture singular}

%We can adapt Castelnuovo's conjecture, its consequences and preliminary results to singular threefolds. 
Interestingly, Castelnuovo's conjecture, its consequences and preliminary results can be adapted to singular threefolds.
Let us see which ones and how.
Let $W$ be an 
%\textit{normal} 
irreducible threefold with \textit{isolated singularities} and let $\mathcal{L}$ be an $r$-dimensional linear system on $W$, where $r\ge 2$, such that the general element $S\in\mathcal{L}$ is a smooth irreducible surface disjoint from the singular points of $W$. Let us take a resolution $f: \widehat{W} \to W$ of the singularities of $W$. Since $f$ is an isomorphism outside the singular points of $W$, we have that the surface $f^{-1}(S)$ is isomorphic to $S$. Furthermore, $f^{-1}(S)$ moves in the linear system $\widehat{\mathcal{L}}:=f^* \mathcal{L}$, which still has $\dim \widehat{\mathcal{L}} = r$. 
%Indeed, by using the \textit{projection formula} (\cite{Hart}, II, Ex. 5.1 (d)) and the normality of $W$, we obtain
%$$h^0(\widehat{W},f^* \mathcal{L}) = h^0 (W,f_* f^* \mathcal{L}) = h^0 (W,\mathcal{L} \otimes f_* \mathcal{O}_{\widehat{W}}) = h^0 (W,\mathcal{L} \otimes \mathcal{O}_{W}) =  h^0 (W,\mathcal{L}).$$
Consequently we have a smooth irreducible threefold $\widehat{W}$ endowed with an $r$-dimensional linear system $\widehat{\mathcal{L}}$, where $r\ge 2$, such that the general element $\widehat{S}\in \widehat{\mathcal{L}}$ is a smooth irreducible surface.
If in addition $W$ is rational and $p_g(S) = p_a(S) = 0$, then $\widehat{W}$ is rational too and $p_g(\widehat{S}) = p_a(\widehat{S}) = 0$. Let $w$ be a general point of $W$: since we may assume that $w$ is not a singular point of $W$, then $\widehat{w}:=f^{-1}(w)$ is still a point of $\widehat{W}$. Furthermore, if $\mathcal{L}_{\bullet}$ is the sublinear system of $\mathcal{L}$ given by the surfaces of $\mathcal{L}$ having a triple point at $w \in W$, then $\widehat{\mathcal{L}}_{\bullet} := f^* \mathcal{L}_{\bullet}$ is the sublinear system of $\widehat{\mathcal{L}}$ given by the surfaces of $\widehat{\mathcal{L}}$ having a triple point at $\widehat{w}\in \widehat{W}$.
Thus, we can adapt Castelnuovo's conjecture to a 
%normal 
rational irreducible threefold $W$ with isolated singularities endowed with an $r$-dimensional linear system $\mathcal{L}$ whose general element is a smooth irreducible surface disjoint from the singular points of $W$, since we can birationally work with the pair $(\widehat{W}, \widehat{\mathcal{L}})$ defined as above.
For completeness, let us state the following results.
%generalizations of previous results.

\begin{theorem}\label{thm: generalization C-E}
Let $W$ be an 
%normal 
irreducible threefold with isolated singularities and let $\mathcal{L}$ be an $r$-dimensional linear system on $W$, where $r\ge 2$, such that the general element is a smooth irreducible surface disjoint from the singular points of $W$. If the elements of $\mathcal{L}$ are big and nef divisors, then a desingularization of $W$ has the same irregularity as a general surface $S\in \mathcal{L}$.
\end{theorem}
\begin{proof}
Let us apply Proposition~\ref{prop:C-E qW=qS} to the pair $(\widehat{W},\widehat{\mathcal{L}})$, constructed as above.
\end{proof}

\begin{theorem}\label{thm: generalization cor C-E}
Let $W$ be an 
%normal 
irreducible threefold with isolated singularities and let $\mathcal{L}$ be an $r$-dimensional linear system on $W$, where $r\ge 3$, such that the general element is a smooth irreducible surface disjoint from the singular points of $W$. If $W$ has regular desingularization and if the intersection of two general surfaces of $\mathcal{L}$ (outside the base locus) is an irreducible curve, then a general element $S\in \mathcal{L}$ is a regular surface.
\end{theorem}
\begin{proof}
Let us apply Proposition~\ref{prop:C-E Sregular} to the pair $(\widehat{W},\widehat{\mathcal{L}})$, constructed as above.
\end{proof}

\begin{theorem}\label{thm: generalization L*}
Let $W$ be an 
%normal 
irreducible threefold with isolated singularities and let $\mathcal{L}_{\bullet}$ be an $r_{\bullet}$-dimensional linear system on $W$, where $r_{\bullet}\ge 3$, such that the general element is an irregular smooth irreducible surface disjoint from the singular points of $W$. If $W$ has regular desingularization, then two general elements $S_{\bullet}$ and $S_{\bullet}'$ of $\mathcal{L}_{\bullet}$ intersect each other (outside the base locus) along a reducible curve. In particular on a fixed $S_{\bullet}$, the components of such a curve are fibres of a pencil of genus $b$ with $0\le b\le q(S_{\bullet})$. Furthermore, by moving the surface $S_{\bullet}$, these component curves give a congruence 
of curves of $W$.
\end{theorem}
\begin{proof}
Let us apply Proposition~\ref{prop:C-E L*} to the pair $(\widehat{W},\widehat{\mathcal{L}}_{\bullet})$, constructed as above.
\end{proof}

We will study a particular type of pair 
$(W, \mathcal{L})$ with the properties seen above.
Let us recall that an \textit{Enriques surface} is a smooth, irreducible surface $S$ with zero irregularity $q(S)=0$ and non-trivial canonical divisor $K_{S}$ such that $2K_{S} \sim 0$.

\begin{definition}\label{def:EF}
A pair $(W,\mathcal{L})$, or simply $W$, is called \textit{Enriques-Fano threefold} if
\begin{itemize}
\item[(i)] $W$ is a normal threefold;
\item[(ii)] $\mathcal{L}$ is a complete linear system of ample Cartier divisors on $W$ such that the general element $S\in \mathcal{L}$ is an Enriques surface;
\item[(iii)] $W$ is not a \textit{generalized cone} over $S$, i.e., $W$ is not obtained by contraction of the negative section on the $\mathbb{P}^1$-bundle $\mathbb{P}(\mathcal{O}_{S}\oplus \mathcal{O}_{S}(S))$ over $S$.
\end{itemize}
\end{definition}

We define the \textit{genus} and the \textit{degree} of an Enriques-Fano threefold $(W,\mathcal{L})$ to be, respectively, the values $p:=\frac{S^{3}}{2}+1$ and $\deg (W) := S^3$, where $S\in \mathcal{L}$. 
Hence $\deg(W) = 2p-2$. 
The linear system $\mathcal{L}$ defines a rational map $\phi_{\mathcal{L}} : W \dashrightarrow \mathbb{P}^{p}$, where $\dim \mathcal{L} = p\ge 2$. Furthermore, the genus $p$ of an Enriques-Fano threefold $(W,\mathcal{L})$ is at most $17$ and the bound is sharp (see \cite{KLM11} and \cite{Pro07}). 

\begin{definition}\label{def:terminal canonical sing}
Let $W$ be a normal variety such that its canonical divisor $K_{W}$ is $\mathbb{Q}$-Cartier.
%$nK_{W}$ is a Cartier divisor for some $n\in \mathbb{N}$ with $n\ge 1$. 
Let 
$f : \widetilde{W} \to W$ be a resolution of the singularities of $W$ and let $\{E_{i}\}_{i\in I}$ be the family of all irreducible exceptional divisors. Since we have
that
$K_{\widetilde{W}} = f^{*}\left(K_{W}\right) + \sum_{i\in I} a_{i} E_{i}$ with $a_i \in \mathbb{Q},$
we say that the singularities of $W$ are \textit{terminal} if $a_{i} > 0$ for all $i$, and 
%we say that 
they are \textit{canonical} if $a_{i} \ge 0$ for all $i$.
\end{definition}

It is known that every Enriques-Fano threefold $(W,\mathcal{L})$ is singular with isolated 
%canonical singularities
singular points (see \cite[Lemma 3.2]{CoMu85}): moreover 
%the bicanonical divisor $2K_{W}$ is a Cartier divisor 
$K_{W}$ is $2$-Cartier and these singularities are canonical (see \cite{Ch96}). Since a general $S\in \mathcal{L}$ is a Cartier divisor on $W$ and it is an Enriques surface, then $S$ is a smooth surface with zero geometric genus and zero arithmetic genus which is disjoint from the singular points of $W$. Furthermore, $W$ has regular desingularization (see \cite[Lemma 4.1]{CDGK20}).
Although the classification of Enriques-Fano threefolds $(W,\mathcal{L})$ still accounts for an open question, examples have been found by several authors.
Fano claimed that Enriques-Fano threefolds exist only for $p=4,6,7,9,13$ (see \cite{Fa38}),
but his classification is incomplete.
Under the assumption that the singularities of $W$ are terminal cyclic quotients, the Enriques-Fano threefolds were classified by Bayle in \cite{Ba94} and, in a similar and independent way, by Sano in \cite{Sa95}. If $W$ appears in the papers of Bayle and Sano, it has genus $2\le p \le 10$ or $p=13$. 
More generally, if an Enriques-Fano threefold has terminal singularities, then it admits a $\mathbb{Q}$\textit{-smoothing}, i.e., it appears as central fibre of a small deformation over the $1$-parameter unit disk such that a general fibre only has cyclic quotient terminal singularities (see \cite[Main Theorem 2]{Mi99}). Hence every Enriques-Fano threefold with only terminal singularities is a limit of the ones contained in the papers of Bayle and Sano. Thus, to complete the classification, 
%one has to consider the case of non-terminal canonical singularities. 
the case of non-terminal canonical singularities ought to be considered.
To date, only a few examples of Enriques-Fano threefolds with non-terminal canonical singularities are known: one of genus $p=9$ found by Knutsen, Lopez and Mu\~{n}oz in \cite{KLM11} and another one of genus $17$ found by Prokhorov in \cite{Pro07}. 
Finally there is an Enriques-Fano threefold of genus $13$, which was mentioned very briefly by Prokhorov (see \cite[Remark 3.3]{Pro07}).

\section{Enriques-Fano threefolds of genus greater than or equal to 13}\label{sec:EFGENUSgreater13}

In order to adapt the ideas of Castelnuovo to the Enriques-Fano threefolds $(W,\mathcal{L})$, we are interested in the ones of genus $ p= \dim \mathcal{L} \ge 13$. Let us list the known Enriques-Fano threefolds of genus $13 \le p \le 17$.

\begin{classical}
Let $\mathcal{S}$ be the linear system of the sextic surfaces of $\mathbb{P}^{3}$ having double points along the six edges of a fixed tetrahedron $T\subset \mathbb{P}^{3}$. 
Up to a change of coordinates, one can consider the tetrahedron $T=\{s_0s_1s_2s_3=0\}$ in $\mathbb{P}^3_{\left[ s_0:s_1:s_2:s_3 \right]}$. In this case $\mathcal{S}$ has equation 
$$\lambda_0s_1^2s_2^2s_3^2+ \lambda_1s_0^2s_2^2s_3^2+\lambda_2s_0^2s_1^2s_3^2+\lambda_3s_0^2s_1^2s_2^2+ s_0s_1s_2s_3Q(s_0,s_1,s_2,s_3)=0,$$
where $\lambda_0,\lambda_1,\lambda_2,\lambda_3\in \mathbb{C}$ and $Q(s_0,s_1,s_2,s_3)=\sum_{i\le j} q_{ij}s_is_j$ is a quadratic form (see \cite[p.635]{GH}). Then $\dim \mathcal{S} = 13$ and a general element of $\mathcal{S}$ has triple points at the vertices of $T$.
The rational map $\nu_{\mathcal{S}} : \mathbb{P}^3 \dashrightarrow \mathbb{P}^{13}$ defined by $\mathcal{S}$ is birational onto the image, which we will denote by $W_F^{13}$; indeed, the linear system $\mathcal{S}$ is very ample outside the tetrahedron $T$, since it contains a sublinear system whose fixed part is given by the tetrahedron $T$ and whose movable part is given by the quadric surfaces of $\mathbb{P}^3$.
If $\mathcal{L}$ denotes the linear system of the hyperplane sections of $W_F^{13}\subset \mathbb{P}^{13}$, then $(W_{F}^{13},\mathcal{L})$ is an Enriques-Fano threefold of genus $p=13$ (see \cite[\S 8]{Fa38}). The threefold $W_{F}^{13}$ is rational by construction and it has eight quadruple points whose tangent cone is a cone over a Veronese surface. 
%Since the threefold $W_F^{13}$ only has terminal singularities (see \cite[Example 1.3]{Re87}), then by \cite[Main Theorem 2]{Mi99} it is limit of the Enriques-Fano threefold of genus $13$ of the Bayle-Sano list, which is the one found in \cite[\S 6.3.2]{Ba94} and in \cite[Theorem 1.1 No.14]{Sa95}. 
We will refer to this Enriques-Fano threefold $(W_{F}^{13},\mathcal{L})$ as the \textit{classical Enriques-Fano threefold}.
\end{classical}

\begin{bayle13}
Let us consider the smooth Fano threefold $X = \mathbb{P}^{1}\times \mathbb{P}^{1}\times \mathbb{P}^{1}$ and the involution $\sigma : X \to  X$ given by 
$$\left[x_0 : x_1\right] \times \left[y_0 : y_1\right] \times \left[z_0 : z_1\right] \mapsto \left[x_0 : -x_1\right] \times \left[y_0 : -y_1\right] \times \left[z_0 : -z_1\right],$$
which has the following eight fixed points
$$p_1' = \left[0 : 1\right] \times \left[1 : 0\right] \times \left[1 : 0\right], \quad
p_1 = \left[1 : 0\right] \times \left[0 : 1\right] \times \left[0 : 1\right],$$
$$p_2' = \left[0 : 1\right] \times \left[0 : 1\right] \times \left[0 : 1\right], \quad 
p_2 = \left[1 : 0\right] \times \left[1 : 0\right] \times \left[1 : 0\right],$$ 
$$p_3 = \left[0 : 1\right] \times \left[1 : 0\right] \times \left[0 : 1\right], \quad
p_3' = \left[1 : 0\right] \times \left[0 : 1\right] \times \left[1 : 0\right],$$ 
$$p_4 = \left[0 : 1\right] \times \left[0 : 1\right] \times \left[1 : 0\right], \quad 
p_4' = \left[1 : 0\right] \times \left[1 : 0\right] \times \left[0 : 1\right].$$ 
Let us take the quotient morphism $\pi : X \to X/ \sigma =: W_{BS}^{13}$, which is defined by the sublinear system of $|-K_X|$ given by the $\sigma$-invariant elements. The threefold $W_{BS}^{13}$ is an Enriques-Fano threefold of genus $13$ and it is endowed by a linear system $\mathcal{L}$ defining an embedding $\phi_{\mathcal{L}} : W_{BS}^{13} \hookrightarrow \mathbb{P}^{13}$ (see \cite[\S 6.3.2]{Ba94} and \cite[Theorem 1.1 No.14]{Sa95}). Furthermore, $W_{BS}^{13}$ has eight singular points, given by the images via $\pi$ of the eight fixed points of $\sigma$. Thanks to the use of Macaulay2 it turns out that the embedding of $W_{BS}^{13}$ in $\mathbb{P}^{13}$ is the classical Enriques-Fano threefold (see Code~\ref{code:bayle13} of Appendix~\ref{app:code}). 
\end{bayle13}

\begin{pro17} 
Let $P$ be the octic Del Pezzo surface given by the image of the anticanonical embedding of $\mathbb{P}^{1}\times \mathbb{P}^{1}$ in $\mathbb{P}^{8}$, which is defined by the linear system of the divisors of bidegree $(2,2)$, i.e.
$$\left[u_0 : u_1\right] \times \left[v_0 : v_1\right]  \mapsto  \left[y_{0,0}:y_{0,1}:y_{0,2},y_{1,0}:y_{1,1}:y_{1,2}:y_{2,0}:y_{2,1}:y_{2,2}\right]$$
where $y_{i,j}:=u_0^{i}u_1^{2-i}v_0^{j}v_1^{2-j}$. Let us consider $\mathbb{P}^{8}$ as the hyperplane $\{x=0\}$ in the $9$-dimensional projective space $\mathbb{P}^9_{\left[y_{0,0}:\dots:y_{2,2}:x\right]}$ and let us take the cone $V$ over $P$ with vertex at $v:=\left[0:\dots :0:1\right]$. The cone $V$ is a Fano threefold with anticanonical divisor $-K_V = 2M$, where $M$ is the class of the hyperplane sections of $V$ (see \cite[Lemma 3.1]{Pro07}). Let $\tau : V \to V$ be the involution defined by $\tau(x)=-x$ and $\tau(y_{i,j})=(-1)^{i+j}y_{i,j}$. This involution has five fixed points $v$, $v_{0,0}$, $v_{0,2}$, $v_{2,0}$, $v_{2,2}$, where $v_{i,j}:=\{x=0,y_{k,l}=0|(k,l)\ne (i,j)\}$. There exists a linear system $\mathcal{Q}$ of quadric sections of $V$ such that $\mathcal{Q}$ is base point free and each member of $\mathcal{Q}$ is $\tau$-invariant:
it is given by the restriction on $V$ of the quadric hypersurfaces of $\mathbb{P}^{9}$ of type
$q_{1}(y_{0,0},y_{0,2},y_{1,1},y_{2,0},y_{2,2})+q_2(y_{0,1},y_{1,0},y_{1,2},y_{2,1},x)=0,$
where $q_1$ and $q_2$ are quadratic homogeneous forms. 
In particular a general member $\widetilde{S}\in \mathcal{Q}$ is smooth, it does not contain any of $v$, $v_{0,0}$, $v_{0,2}$, $v_{2,0}$, $v_{2,2}$ and the action of $\tau$ on $\widetilde{S}$ is fixed point free. Since $\mathcal{Q} \subset |2M|=|-K_V|$, then $\widetilde{S}$ is a K3 surface. Let $\pi : V \to  V/\tau =: W_{P}^{17}$ be the quotient morphism and let $S:= \pi(\widetilde{S})= \widetilde{S} /\tau$. Then $S$ is a smooth Enriques surface and if we set $\mathcal{L}:=|\mathcal{O}_{W_{P}^{17}}(S)|$, we have that $(W_{P}^{17},\mathcal{L})$ is an Enriques-Fano threefold of genus 17 (see \cite[Proposition 3.2]{Pro07}). The threefold $W_{P}^{17}$ is unirational by construction and it has five singular points, given by the images via $\pi$ of the five fixed points of $\tau$. Understanding whether $W_P^{17}$ is rational remains an open question. We will refer to this Enriques-Fano threefold $(W_{P}^{17},\mathcal{L})$ as the \textit{Prokhorov-Enriques-Fano threefold of genus 17}.
\end{pro17}

\begin{pro13} 
Let $S_6 \subset \mathbb{P}^{6}$ be the sextic Del Pezzo surface, which is the image of $\mathbb{P}^{2}$ via the rational map defined by the linear system of the cubic plane curves passing through three fixed points in general position. Let us consider $\mathbb{P}^{6}$ as a hyperplane $H$ in a $7$-dimensional projective space $\mathbb{P}^7$ and let us take the cone $V$ over $S_6$ with vertex at a general point $v\in\mathbb{P}^7$. Prokhorov claimed that one can construct an Enriques-Fano threefold $(W_{P}^{13},\mathcal{L})$ of genus $13$, in a similar way as for the Prokhorov-Enriques-Fano threefold of genus $17$ (see \cite[Remark 3.3]{Pro07}). 
%Also in this case $W_{P}^{13}$ is \textit{unirational}, while its rationality is an open question. 
We will refer to this object as the \textit{Prokhorov-Enriques-Fano threefold of genus 13}.
\end{pro13}

In \S~\ref{sec: classical EF3-fold} we will apply the ideas of Castelnuovo to the classical Enriques-Fano threefold $(W_F^{13},\mathcal{L})$: we will see that case (\ref{castelnuovoA}) of Castelnuovo's conjecture occurs for $(W_F^{13},\mathcal{L})$ (see Theorem~\ref{thm:MAINthmCONJ}). 
This will lead to the claim that case (\ref{castelnuovoA}) of Castelnuovo's conjecture also occurs for the Prokhorov-Enriques-Fano threefold of genus $17$ (see Corollary~\ref{cor:castelnuovop17}). 
We observe that it actually makes sense to ask ourselves about the link between the arguments of Castelnuovo and the Prokhorov-Enriques-Fano threefolds, since the Remark~\ref{rem:rationality congruence if W rational} also holds for a unirational variety.

\section{Castelnuovo's conjecture for the classical Enriques-Fano threefold} \label{sec: classical EF3-fold}

Let us consider the classical Enriques-Fano threefold $(W=W_F^{13},\mathcal{L})$. We aim to study the sublinear system $\mathcal{L}_{\bullet}\subset \mathcal{L}$ of the hyperplane sections of $W$ with triple point at a general point $w\in W$.

We recall that $W$ is the image of $\mathbb{P}^{3}$ via the birational map $\nu_{\mathcal{S}} : \mathbb{P}^{3} \dashrightarrow W\subset \mathbb{P}^{13}$, defined by the linear system $\mathcal{S}$ of the sextic surfaces having double points along the edges of a tetrahedron $T$ (see \cite[\S 8]{Fa38}).
Let us fix the tetrahedron $T\subset \mathbb{P}^3$ with vertices $v_0$, $v_1$, $v_2$, $v_3$ and faces $f_0$, $f_1$, $f_2$, $f_3$, such that $f_i$ is the face opposite to the vertex $v_i$, and let us denote the edges by $l_{ij}:=f_i\cap f_j$, for $0\le i<j\le 3$.
By definition, for each surface $S\in \mathcal{L}$ there is a unique sextic surface $\Sigma\in \mathcal{S}$ such that $S = \nu_{\mathcal{S}}(\Sigma)$. Hence, if we take a surface $S_{\bullet}\in \mathcal{L}_{\bullet}\subset \mathcal{L}$, there must exist a unique sextic surface $\Sigma_{\bullet}\in \mathcal{S}$ such that $S_{\bullet} = \nu_{\mathcal{S}}(\Sigma_{\bullet})$. This surface $\Sigma_{\bullet}$ is a particular surface of $\mathcal{S}$, which has triple point at the point $p\in \mathbb{P}^{3}$ such that $w = \nu_{\mathcal{S}}(p)$.
We can so represent the linear system $\mathcal{L}_{\bullet}$ on $W$ via the sublinear system $\mathcal{S}_{\bullet}\subset \mathcal{S}$ on $\mathbb{P}^{3}$ given by the sextic surfaces of $\mathbb{P}^{3}$ double along the six edges of the tetrahedron $T$
and 
triple at the point $p \in \mathbb{P}^{3}$ such that $\nu_{\mathcal{S}}(p) = w$. Since $w$ is a general point of $W$, we may consider $p$ as a general point of $\mathbb{P}^{3}$. A priori 
we have that $r_{\bullet} := \dim \mathcal{S}_{\bullet} \ge \dim \mathcal{S} - 10 = 13-10 = 3$ and that the linear system $\mathcal{S}_{\bullet}$ defines a rational map $\nu_{\bullet} : \mathbb{P}^{3} \dashrightarrow \mathbb{P}^{r_{\bullet}}$.

Let us take the plane $\left\langle p,l_{0i} \right\rangle$ generated by the point $p$ and the edge $l_{0i}$, for a fixed index $1\le i\le 3$. If $1\le j<k\le 3$ with $j,k\ne i$, then the edges $l_{0i}$ and $l_{jk}$ are disjoint lines of $\mathbb{P}^{3}$; so the plane $\left\langle p,l_{0i}\right\rangle$ and the line $l_{jk}$ intersect at a point, outside $l_{0i}$. Let $r_i$ be the line joining this point and the point $p$, i.e. $r_i:= \left\langle p,\left\langle p,l_{0i}\right\rangle\cap l_{jk} \right\rangle$.

\begin{proposition}\label{prop:ri contained in base locus}
Let $\mathcal{S}_{\bullet}$ be the linear system on $\mathbb{P}^{3}$ given by the sextic surfaces of $\mathbb{P}^{3}$ double along the six edges of the tetrahedron $T$
and 
triple at the general point $p \in \mathbb{P}^{3}$. Then the three lines $r_1$, $r_2$, $r_3$
are contained in the base locus of $\mathcal{S}_{\bullet}$. 
\end{proposition}
\begin{proof}
Assume the contrary. Let us take a surface $\Sigma_{\bullet}\in \mathcal{S}_{\bullet}$ and let us fix $1\le i,j,k \le 3$ with $j<k$ and $j,k\ne i$. By Bezout's Theorem, $\Sigma_{\bullet}\cap r_i$ is given by $6$ points. Since $r_i$ is a line of the plane $\left\langle p, l_{0i} \right\rangle$, then it intersects the line $l_{0i}$ at a point. Thus, $r_i$ is a line joining the triple point $p$ of $\Sigma_{\bullet}$ and two particular double points of $\Sigma_{\bullet}$, each one lying in one of the two opposite disjoint edges $l_{0i}$ and $l_{jk}$. Then $\Sigma_{\bullet}\cap r_i$ contains at least $3+2+2 = 7$ points, counted with multiplicity. This is a contradiction, so $r_i \subset \Sigma_{\bullet}$. 
\end{proof}

Let us denote by $\mathcal{A}$ the two-dimensional linear system of the planes of $\mathbb{P}^{3}$ passing through the point $p$. On a general plane $\alpha\in\mathcal{A}$ we can construct a cubic plane curve $\gamma_{\alpha}$ with node at $p$ and passing through the six points given by the intersection of 
$\alpha$ with the six edges of the tetrahedron $T$. Let us denote these six points by $A_{ij} := \alpha \cap l_{ij}$ for $0\le i<j\le 3$.

\begin{lemma}\label{lem:unique gamma_alpha}
On a general plane $\alpha \subset \mathbb{P}^{3}$ passing through the point $p$, there is a unique cubic plane curve $\gamma_{\alpha}$, defined as above.
\end{lemma}
\begin{proof}
Let $\mathfrak{g}$ be the linear system of the cubic plane curves on $\alpha$ passing through the six points $\{A_{ij}|0\le i<j \le 3\}$ and having a node at $p$. If the six fixed points had been general, we would have imposed $\frac{2\cdot 3}{2}+\sum_{i=1}^{6} 1 = 9$ independent conditions. In our case the points
$\{A_{ij}|0\le i<j \le 3\}$ are not in general position: indeed, they are the vertices of a complete quadrilateral whose edges are the intersection of the plane $\alpha$ with the four faces of the tetrahedron $T$. Hence $\dim \mathfrak{g}\ge \binom{3+2}{2}- 9 -1= 0$. 
We want to show that the equality holds. 
In order to do it, we take the blow-up $bl : \widetilde{\alpha} \to \alpha$ of the plane $\alpha$ at the points $\{A_{ij}|0\le i<j \le 3\}\cup p$, by denoting the exceptional divisors by $e_{ij} = bl^{-1}(A_{ij})$ and $e_p = bl^{-1}(p)$, for $1\le i<j\le 3$. If we denote by $\ell$ the strict transform of a general line of $\alpha$, then the strict transform of a general $\gamma_{\alpha} \in \mathfrak{g}$ is
$\widetilde{\gamma}_{\alpha} \sim 3\ell - 2e_p -\sum_{0\le i<j \le 3}e_{ij}.$
By the generality of the point $p\in \mathbb{P}^{3}$, we may assume that the five points $p$, $A_{02}$, $A_{13}$, $A_{03}$, $A_{12}$ are in general position, since no three of them are collinear. So we can consider the unique irreducible conic $\delta$ passing through $p$, $A_{02}$, $A_{13}$, $A_{03}$, $A_{12}$, as in Figure~\ref{fig:quadrilateralwithconic}, which has strict transform
$\widetilde{\delta} \sim 2\ell - e_p-e_{02}-e_{13}-e_{03}-e_{12}.$
\begin{figure}[h]
\centering
\includegraphics[scale=0.5]{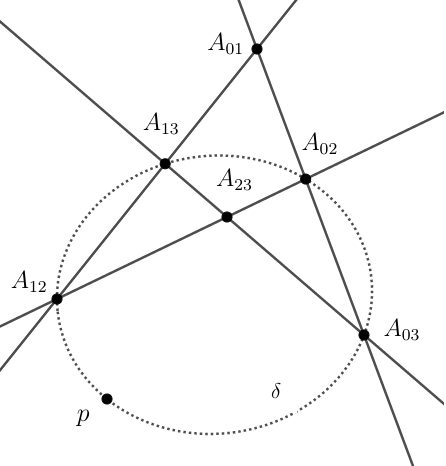}
\caption{\label{fig:quadrilateralwithconic}\scriptsize{The complete quadrilateral on $\alpha$ with vertices at the points $\{A_{ij}|0\le i<j \le 3\}$ and the conic $\delta$ uniquely determined by the points $p, A_{02}, A_{13}, A_{03}, A_{12}$.}}
\end{figure}
Since $\widetilde{\gamma}_{\alpha}\cdot \widetilde{\delta} = 0$, then we have the following exact sequence
$$0 \to \mathcal{O}_{\widetilde{\alpha}}(l-e_p-e_{01}-e_{23}) \to \mathcal{O}_{\widetilde{\alpha}}(\widetilde{\gamma}_{\alpha}) \to \mathcal{O}_{\widetilde{\delta}} \to 0.$$
Obviously $h^0 (\widetilde{\alpha} , \mathcal{O}_{\widetilde{\alpha}}(l-e_p-e_{01}-e_{23}) ) = 0$, since the three points $p, A_{01}, A_{23}$ are not collinear, by the generality of the point $p$ again.
Hence
$h^{0}(\widetilde{\alpha}, \mathcal{O}_{\widetilde{\alpha}}(\widetilde{\gamma}_{\alpha})) \le h^0(\mathcal{O}_{\widetilde{\delta}}) = 1$
and $\dim \mathfrak{g}= h^{0}(\widetilde{\alpha}, \mathcal{O}_{\widetilde{\alpha}}(\widetilde{\gamma}_{\alpha})) -1=0$.
\end{proof}

\begin{lemma}\label{lem:contraction of congruence}
Let $\mathcal{S}_{\bullet}$ be the linear system on $\mathbb{P}^{3}$ given by the sextic surfaces of $\mathbb{P}^{3}$ double along the six edges of the tetrahedron $T$
and 
triple at the general point $p \in \mathbb{P}^{3}$. Then the rational map $\nu_{\bullet} : \mathbb{P}^{3} \dashrightarrow \mathbb{P}^{r_{\bullet}}$ defined by $\mathcal{S}_{\bullet}$ contracts the cubic plane curves $\gamma_{\alpha}$, constructed as above. 
\end{lemma}
\begin{proof}
By Bezout's Theorem, a general element $\Sigma_{\bullet}\in\mathcal{S}_{\bullet}$ intersects a cubic plane curve $\gamma_{\alpha}$ in $6\cdot 3 = 18$ points. Furthermore, $\Sigma_{\bullet}$ and $\gamma_{\alpha}$ have in common, in the base locus of $\mathcal{S}_{\bullet}$, the point $p$ (which is a triple point for $\Sigma_{\bullet}$ and a node for $\gamma_{\alpha}$) and the six points $\{A_{ij}|0\le i<j \le 3\}$ (which are nodes for $\Sigma_{\bullet}$ and simple points for $\gamma_{\alpha}$). Hence, outside the base locus, we have that $\Sigma_{\bullet}\cap \gamma_{\alpha}$ is given by $6\cdot 3 - \sum_{i=1}^{6}2\cdot 1 - 3\cdot 2 = 0$ points. So $\gamma_{\alpha}$ is contracted to a point by $\nu_{\bullet}: \mathbb{P}^3 \dashrightarrow \mathbb{P}^{r_{\bullet}}$.
\end{proof}

\begin{remark}\label{rem:generalfibreM2}
Thanks to a computational analysis via Macaulay2, one can see that the general fibre of the rational map $\nu_{\bullet} : \mathbb{P}^{3} \dashrightarrow \mathbb{P}^{r_{\bullet}}$ defined by $\mathcal{S}_{\bullet}$ is a cubic plane curve $\gamma_{\alpha}$ (see Code~\ref{code:castelnuovoClassical} of Appendix).
\end{remark}

\begin{proposition}
The cubic plane curves $\gamma_{\alpha}$, defined as above, give a congruence $\mathcal{V}$ of curves of $\mathbb{P}^3$.
\end{proposition}
\begin{proof}
By Lemma~\ref{lem:unique gamma_alpha} we have that the set of the cubic plane curves $\gamma_{\alpha}$ is a $2$-dimensional family $\mathcal{V}$. In particular $\mathcal{V}$ is birationally parametrized by the same projective plane $\mathbb{P}^{2}$ parametrizing the planes passing through $p$. It remains to show that, given a general point $p' \in \mathbb{P}^{3}$, there is a unique curve of $\mathcal{V}$ passing through it. By Lemma~\ref{lem:contraction of congruence} and Remark~\ref{rem:generalfibreM2} we have that the curves of $\mathcal{V}$ are the general fibres of the rational map $\nu_{\bullet}: \mathbb{P}^{3} \dashrightarrow \mathbb{P}^{r_{\bullet}}$ defined by $\mathcal{S}_{\bullet}$. Hence $\nu_{\bullet}^{-1}(\nu_{\bullet}(p'))$ is the unique curve of $\mathcal{V}$ passing through $p'$.
\end{proof}

\begin{corollary}\label{cor: imageS* = surface}
Let $\mathcal{S}_{\bullet}$ be the linear system on $\mathbb{P}^{3}$ given by the sextic surfaces of $\mathbb{P}^{3}$ double along the six edges of the tetrahedron $T$
and 
triple at the general point $p \in \mathbb{P}^{3}$. Then the image of $\mathbb{P}^{3}$ via the rational map $\nu_{\bullet} : \mathbb{P}^{3} \dashrightarrow \mathbb{P}^{r_{\bullet}}$ defined by $\mathcal{S}_{\bullet}$ is a surface $\Delta \subset \mathbb{P}^{r_{\bullet}}$.
\end{corollary}
\begin{proof}
Let $\Delta$ be the image of $\mathbb{P}^{3}$ via $\nu_{\bullet}$. By Lemma~\ref{lem:contraction of congruence} and Remark~\ref{rem:generalfibreM2}, the general fibre of $\nu_{\bullet}$ is a cubic plane curve, so we have $\dim \Delta = 3-1 = 2$.
\end{proof}

Let us now pay attention to a particular surface of $\mathbb{P}^{3}$. Let us consider the linear system $\mathfrak{c}$ on $\mathbb{P}^{2}$ given by the cubic plane curves passing through the six vertices of a complete quadrilateral. The image of $\mathbb{P}^{2}$ via the rational map defined by $\mathfrak{c}$ is a special Del Pezzo surface of degree $3$, which is called \textit{Cayley cubic surface} (see \cite[\S 9.2.2]{Dolg12}). This surface has four singular points whose tangent cone is a quadric cone: we will refer to these singularities as \textit{nodes}. The four nodes of the Cayley cubic surface are given by the image of the four edges of the fixed complete quadrilateral.

\begin{theorem}\label{thm: image Cayley cubic}
Let $\mathcal{S}_{\bullet}$ be the linear system on $\mathbb{P}^{3}$ given by the sextic surfaces of $\mathbb{P}^{3}$ double along the six edges of the tetrahedron $T$
and 
triple at the general point $p \in \mathbb{P}^{3}$. Then the image of $\mathbb{P}^{3}$ via the rational map $\nu_{\bullet} : \mathbb{P}^{3} \dashrightarrow \mathbb{P}^{r_{\bullet}}$ defined by $\mathcal{S}_{\bullet}$ is a Cayley cubic surface $\Delta \subset \mathbb{P}^{3}$. Thus, $r_{\bullet} = 3$.
\end{theorem}
\begin{proof}
Let us take a general element $\alpha\in \mathcal{A}$, i.e. a general plane passing through $p$. If we restrict the linear system $\mathcal{S}_{\bullet}$ to this plane, we obtain the linear system $\mathfrak{s}$ on $\alpha$ of the sextic plane curves with triple point at $p$ and nodes at the six points $\{A_{ij}|0\le i<j \le 3\}$. The plane $\alpha$ and a general fibre of $\nu_{\bullet}$ intersect, outside the base locus of $\mathcal{S}_{\bullet}$, at a single point: indeed, the general fibre of $\nu_{\bullet}$ is a cubic plane curve $\gamma_{\alpha'}$ contained in a plane $\alpha' \in \mathcal{A}$, where $\alpha'\neq \alpha $; so we have that $\alpha$ intersects $\gamma_{\alpha'}$, outside the base locus of $\mathcal{S}_{\bullet}$, at $1\cdot 3 - 1\cdot 2 = 1$ point. Then the linear system $\mathfrak{s}$ defines the rational map $\nu_{\bullet}|_{\alpha} : \alpha \cong \mathbb{P}^{2} \dashrightarrow \mathbb{P}^{r_{\bullet}}$ which is generically $1:1$.
In the following we will see that, by applying three quadratic transformations, we obtain, from $\mathfrak{s}$, the linear system $\mathfrak{c}$ of the cubic plane curves passing through the six vertices of a complete quadrilateral. Thus, the image of $\alpha$ via $\nu_{\bullet}|_{\alpha}$ is the image of $\mathbb{P}^{2}$ via the rational map defined by $\mathfrak{c}$, that is a Cayley cubic surface. 
By Corollary~\ref{cor: imageS* = surface}, this is the image $\Delta$ of $\mathbb{P}^{3}$ via $\nu_{\bullet} : \mathbb{P}^3 \dashrightarrow \mathbb{P}^{r_{\bullet}}$. Hence $r_{\bullet} = 3$.

Let us recall that the four faces of the tetrahedron $T$ intersect the plane $\alpha$ along four lines: the line $\left\langle A_{01}, A_{02}, A_{03} \right\rangle$ passing through $A_{01}$, $A_{02}$, $A_{03}$, the line $\left\langle A_{01}, A_{12}, A_{13} \right\rangle$ passing through $A_{01}$, $A_{12}$, $A_{13}$, the line $\left\langle A_{02}, A_{12}, A_{23} \right\rangle$ passing through $A_{02}$, $A_{12}$, $A_{23}$ and the line $\left\langle A_{03}, A_{13}, A_{23} \right\rangle$ passing through $A_{03}$, $A_{13}$, $A_{23}$. These four lines are the edges of a complete quadrilateral $Q_A$ with six vertices at the points $\{A_{ij}|0\le i<j \le 3\}$.
Hence $\mathfrak{s}$ is the linear system of the sextic plane curves triple at $p$ e double at the six vertices of $Q_A$. 
Let us consider the quadratic trasformation $q_{p,A_{12},A_{03}} : \mathbb{P}^{2} \dashrightarrow \mathbb{P}^{2}$ given by the linear system of the conics passing through the three points $p,A_{12},A_{03}$. Let $B_{23}$, $B_{13}$, $B_{01}$, $B_{02}$ be the images of the points $A_{23}$, $A_{13}$, $A_{01}$, $A_{02}$. We have that each of the lines $\left\langle p, A_{12} \right\rangle$, $\left\langle p, A_{03} \right\rangle$ and $\left\langle A_{12}, A_{03} \right\rangle$ is contracted by $q_{p,A_{12},A_{03}}$ to a point, denoted respectively by $B_{03}$, $B_{12}$ and $p'$. Furthermore, the four edges of the complete quadrilateral $Q_A$ are sent to the four edges of a new complete quadrilateral $Q_B$ with six vertices at the points $\{B_{ij}|0\le i<j \le 3\}$: in particular we have that
$$q_{p,A_{12},A_{03}}(\left\langle A_{01}, A_{02}, A_{03}\right\rangle ) = \left\langle B_{01}, B_{02}, B_{03}\right\rangle ,$$
$$q_{p,A_{12},A_{03}}(\left\langle A_{01}, A_{12}, A_{13}\right\rangle ) = \left\langle B_{01}, B_{12}, B_{13}\right\rangle ,$$
$$q_{p,A_{12},A_{03}}(\left\langle A_{02}, A_{12}, A_{23}\right\rangle ) = \left\langle B_{02}, B_{12}, B_{23}\right\rangle ,$$
$$q_{p,A_{12},A_{03}}(\left\langle A_{03}, A_{13}, A_{23}\right\rangle ) = \left\langle B_{03}, B_{13}, B_{23}\right\rangle .$$
Thus, the linear system $\mathfrak{s}$ of the sextic plane curves triple at the point $p$ and double at the six points $\{A_{ij}|0\le i<j \le 3\}$ is transformed in the linear system $\mathfrak{q}_{5}$ of the quintic plane curves double at $p'$, $B_{23}$, $B_{13}$, $B_{01}$, $B_{02}$ and passing through $B_{12}$ and $B_{03}$.
Let us consider the quadratic trasformation $q_{p',B_{23},B_{01}} : \mathbb{P}^{2} \dashrightarrow \mathbb{P}^{2}$ given by the linear system of the conics passing through the three points $p'$, $B_{23}$, $B_{01}$. Let $C_{13}$, $C_{12}$, $C_{02}$, $C_{03}$ be the images of the points $B_{13}$, $B_{12}$, $B_{02}$, $B_{03}$. We have that each of the lines $\left\langle p', B_{23} \right\rangle$, $\left\langle p', B_{01} \right\rangle$ and $\left\langle B_{23}, B_{01} \right\rangle$ is contracted by $q_{p',B_{23},B_{01}}$ to a point, denoted respectively by $C_{01}$, $C_{23}$ and $p''$. Furthermore, the four edges of the complete quadrilateral $Q_B$ are sent to the four edges of a new complete quadrilateral $Q_C$ with six vertices at the points $\{C_{ij}|0\le i<j \le 3\}$, in the following way:
$$q_{p',B_{23},B_{01}}(\left\langle B_{01}, B_{02}, B_{03}\right\rangle ) = \left\langle C_{01}, C_{02}, C_{03}\right\rangle ,$$ 
$$q_{p',B_{23},B_{01}}(\left\langle B_{01}, B_{12}, B_{13}\right\rangle ) = \left\langle C_{01}, C_{12}, C_{13}\right\rangle ,$$
$$q_{p',B_{23},B_{01}}(\left\langle B_{02}, B_{12}, B_{23}\right\rangle ) = \left\langle C_{02}, C_{12}, C_{23}\right\rangle ,$$
$$q_{p',B_{23},B_{01}}(\left\langle B_{03}, B_{13}, B_{23}\right\rangle ) = \left\langle C_{03}, C_{13}, C_{23}\right\rangle .$$
As a result, the linear system $\mathfrak{q}_5$ of the quintic plane curves double at $p'$, $B_{23}$, $B_{13}$, $B_{01}$, $B_{02}$ and passing through $B_{12}$ and $B_{03}$ is transformed in the linear system $\mathfrak{q}_{4}$ of the quartic plane curves double at $C_{13}$ and $C_{02}$ and passing through $p''$, $C_{23}$, $C_{12}$, $C_{01}$, $C_{03}$.
Let us consider the quadratic trasformation $q_{p'',C_{13},C_{02}} : \mathbb{P}^{2} \dashrightarrow \mathbb{P}^{2}$ given by the linear system of the conics passing through the three points $p''$, $C_{13}$, $C_{02}$. Let $D_{23}$, $D_{12}$ $D_{01}$, $D_{03}$ be the images of the points $C_{23}$, $C_{12}$, $C_{01}$, $C_{03}$. We have that each of the lines $\left\langle p'', C_{13} \right\rangle$, $\left\langle p'', C_{02} \right\rangle$ and $\left\langle C_{13}, C_{02} \right\rangle$ are contracted by $q_{p'',C_{13},C_{02}}$ to a point, denoted respectively with $D_{02}$, $D_{13}$ and $p'''$. Furthermore, the four edges of the complete quadrilateral $Q_C$ are sent to the four edges of a new complete quadrilateral $Q_D$ with six vertices the points $\{D_{ij}|0\le i<j \le 3\}$:
$$q_{p'',C_{13},C_{02}}(\left\langle C_{01}, C_{02}, C_{03}\right\rangle ) = \left\langle D_{01}, D_{02}, D_{03}\right\rangle ,$$
$$q_{p'',C_{13},C_{02}}(\left\langle C_{01}, C_{12}, C_{13}\right\rangle ) = \left\langle D_{01}, D_{12}, D_{13}\right\rangle ,$$
$$q_{p'',C_{13},C_{02}}(\left\langle C_{02}, C_{12}, C_{23}\right\rangle ) = \left\langle D_{02}, D_{12}, D_{23}\right\rangle ,$$
$$q_{p'',C_{13},C_{02}}(\left\langle C_{03}, C_{13}, C_{23}\right\rangle ) = \left\langle D_{03}, D_{13}, D_{23}\right\rangle .$$
Then the linear system $\mathfrak{q}_4$ of the quartic plane curves double at $C_{13}$ and $C_{02}$ and passing through $p'''$, $C_{23}$, $C_{12}$, $C_{01}$, $C_{03}$ is transformed in the linear system $\mathfrak{c}$ of the cubic plane curves passing through $\{D_{ij}|0\le i<j \le 3\}$, which are the six vertices of a complete quadrilateral $Q_D$.
\end{proof}

\begin{corollary}\label{cor:onlybasecurves edgesT and ri}
Let $\mathcal{S}_{\bullet}$ be the linear system on $\mathbb{P}^{3}$ given by the sextic surfaces of $\mathbb{P}^{3}$ double along the six edges of the tetrahedron $T$
and 
triple at the general point $p \in \mathbb{P}^{3}$. The only base curves of $\mathcal{S}_{\bullet}$ are the six edges of $T$ and the lines $r_1$, $r_2$, $r_3$.
\end{corollary}
\begin{proof}
Let $\Delta$ be the image of $\mathbb{P}^{3}$ via the rational map $\nu_{\bullet}$. Two of its general hyperplane sections intersect each other at $\deg \Delta = 3$ points (see Theorem~\ref{thm: image Cayley cubic}). Let us consider the preimages of these two curves: they are two elements $\Sigma_{\bullet}$ and $\Sigma_{\bullet}'$ of $\mathcal{S}_{\bullet}$, intersecting, outside the base locus of $\mathcal{S}_{\bullet}$, along a nonic curve. Indeed, the intersection of $\Sigma_{\bullet}$ and $\Sigma_{\bullet}'$, outside the base locus of $\mathcal{S}_{\bullet}$, is given by the union of $\deg \Delta = 3$ fibres of $\nu_{\bullet}$, which are cubic plane curves (see Lemma~\ref{lem:contraction of congruence} and Remark~\ref{rem:generalfibreM2}).
The base locus of $\mathcal{S}_{\bullet}$ contains the six edges of $T$ and the three lines $r_1,r_2,r_3$ (see Proposition~\ref{prop:ri contained in base locus}). If another curve existed in the base locus of $\mathcal{S}_{\bullet}$, then $\Sigma_{\bullet}$ would intersect $\Sigma_{\bullet}'$, outside it, along a curve of degree less than $9$, and so $\deg \Delta < 3$, which is a contradiction.
\end{proof}

By using the notations of the proof of Theorem~\ref{thm: image Cayley cubic}, we have the following facts.

\begin{proposition}\label{prop:nodesDelta}
Let $\mathcal{S}_{\bullet}$ be the linear system on $\mathbb{P}^{3}$ given by the sextic surfaces of $\mathbb{P}^{3}$ double along the six edges of the tetrahedron $T$ and triple at the general point $p \in \mathbb{P}^{3}$. Let $\Delta \subset \mathbb{P}^3$ be the Cayley cubic surface given by the image of the rational map $\nu_{\bullet} : \mathbb{P}^{3} \dashrightarrow \mathbb{P}^{3}$ defined by $\mathcal{S}_{\bullet}$. Then the four nodes of $\Delta$ are given by the image via $\nu_{\bullet}$ of the four faces of the tetrahedron $T$. 
\end{proposition}
\begin{proof}
The faces of $T$ intersect a general plane $\alpha\in \mathcal{A}$ along the four edges of the complete quadrilateral $Q_A$. The edges of $Q_A$ are sent by $\mathfrak{s}$ to the edges of $Q_{B}$, which are mapped by $\mathfrak{q}_{5}$ to the edges of $Q_{C}$, which are transformed by $\mathfrak{q}_{4}$ in the edges of $Q_D$, which are finally sent by $\mathfrak{c}$ to the four singular points of $\Delta$. 
\end{proof}

Let us consider the lines $s_i:= \left\langle p,v_i \right\rangle$ joining the point $p\in \mathbb{P}^{3}$ and the vertex $v_i$ of the tetrahedron $T$, for $0\le i\le 3$.

\begin{corollary}
Let $\mathcal{S}_{\bullet}$ be the linear system on $\mathbb{P}^{3}$ given by the sextic surfaces of $\mathbb{P}^{3}$ double along the six edges of the tetrahedron $T$ and triple at the general point $p \in \mathbb{P}^{3}$. Let $\Delta \subset \mathbb{P}^3$ be the Cayley cubic surface given by the image of the rational map $\nu_{\bullet} : \mathbb{P}^{3} \dashrightarrow \mathbb{P}^{3}$ defined by $\mathcal{S}_{\bullet}$. Then the four lines $s_0$, $s_1$, $s_2$, $s_3$ are sent via $\nu_{\bullet}$ to the four nodes of $\Delta\subset\mathbb{P}^{3}$.
\end{corollary}
\begin{proof}
By Bezout's Theorem, a general sextic surface $\Sigma_{\bullet}\in \mathcal{S}_{\bullet}$ intersects each of the four lines at $6$ points. We also observe that $\Sigma_{\bullet}$ and each of these lines have in common, in the base locus of $\mathcal{S}_{\bullet}$, the point $p$ and a vertex of $T$, which are triple points for $\Sigma_{\bullet}$. Hence, outside the base locus, we have that $\Sigma_{\bullet}\cap s_i$ is given by $6 - 3 - 3 = 0$ points, for all $0\le i\le 3$. So the four lines $s_0$, $s_1$, $s_2$, $s_3$ are contracted by $\nu_{\bullet}$ to four points. Let us fix now $0\le i\le 3$. We have that $s_i$ intersects at a point the face of $T$ opposite to the vertex $v_i$. Hence the point to which the line $s_i$ is sent by $\nu_{\bullet}$ is the same point to which 
%the opposite face to $v_i$
the face $f_i$
is sent by $\nu_{\bullet}$, that is one of the four nodes of $\Delta$ (see Proposition~\ref{prop:nodesDelta}). 
\end{proof}

The surfaces of the linear system $\mathcal{S}_{\bullet}$ will now be the subject of our analysis. First let us recall some facts about the surfaces of the linear system $\mathcal{S}$.

\begin{definition}\label{def:ordinarysingularities}
A surface of $\mathbb{P}^3$ has \textit{ordinary singularities} if it has at most the following singularities: a curve $\gamma$ of double points (that are generically the transverse intersection of two branches) with at most finitely many pinch points, and with $\gamma$ having at most finitely many triple points as singularities, with three independent tangent lines, which are triple points also for the surface. 
\end{definition}

\begin{remark}\label{rem:classicalcase}
Let us blow-up $\mathbb{P}^3$ at the vertices of $T$, obtaining a smooth threefold $Y'$ and a birational morphism $bl' : Y'\to \mathbb{P}^{3}$ with exceptional divisors
$E_i := (bl')^{-1}(v_i)$ for $0\le i \le 3$. 
If $H$ denotes the pullback on $Y'$ of the hyperplane class on $\mathbb{P}^{3}$,
the strict transform of an element of $\mathcal{S}$ is linearly equivalent to $6H-3\sum_{i=0}^{3}E_i$.
Let us blow-up the strict transforms $\widetilde{l}_{ij}$ of the edges of $T$, for $0\le i<j\le 3$: we obtain a smooth threefold $Y''$ and a birational morphism $bl'' : Y'' \to Y'$ with exceptional divisors 
$F_{ij}:=(bl'')^{-1}(\widetilde{l}_{ij})$.
Let $\Sigma''$ be the strict transform on $Y''$ of a general element $\Sigma \in \mathcal{S}$: it is linearly equivalent to $6H-3\sum_{i=0}^3 \widetilde{E}_i-2\sum_{0\le i < j \le 3} F_{ij}$, where $\widetilde{E}_i$ is the strict transform of $E_i$, for $0\le i\le 3$, and $H$ denotes the pullback $bl''^* H$, by abuse of notation. 
%Furthermore 
We have that
$\Sigma''$ is smooth, since it is the blow-up of a surface $\Sigma\in\mathcal{S}$ with ordinary singularities along its singular curves (see \cite[p.620-621]{GH}).
Let us take the following exact sequence
$$0 \to \mathcal{O}_{Y''}(K_{Y''})  \to \mathcal{O}_{Y''}(K_{Y''} + \Sigma'') \to \mathcal{O}_{\Sigma''}(K_{\Sigma''}) \to 0,$$
where $K_{Y''}+\Sigma''\sim 2H - \sum_{i=0}^3\widetilde{E}_i - \sum_{0\le i<j \le 3}F_{ij}$ (see \cite[p.187]{GH}).
We have that $h^{i=0,1,2}(Y'' ,\mathcal{O}_{Y''}(K_{Y''})) = 0$ and $h^3(Y'',\mathcal{O}_{Y''}(K_{Y''})) = 1$ by Serre Duality, since $Y''$ is a rational smooth threefold by construction; furthermore, we have that
$$h^0(\Sigma'', \mathcal{O}_{\Sigma''}(K_{\Sigma''})) = p_g(\Sigma'') =0,\, h^1(\Sigma'', \mathcal{O}_{\Sigma''}(K_{\Sigma''})) =  h^1(\Sigma'', \mathcal{O}_{\Sigma''}) = q(\Sigma'') = 0$$ 
and $h^2(\Sigma'' , \mathcal{O}_{\Sigma''}(K_{\Sigma''})) =  h^0(\Sigma'' , \mathcal{O}_{\Sigma''})= 1$ by Serre Duality, since it is known that the desingularization of a sextic surface in $\mathcal{S}$ is birational to an Enriques surface (see \cite[p.275]{CoDo89}).
So we obtain $h^0(Y'',\mathcal{O}_{Y''}(K_Y''+\Sigma'')) = h^0(\Sigma'', \mathcal{O}_{\Sigma''}(K_{\Sigma''})) = 0$, i.e. there are no quadric surfaces of $\mathbb{P}^{3}$ containing the edges of $T$. 
We also have that 
$h^{1}(Y'',\mathcal{O}_{Y''}(K_Y''+\Sigma'')) = h^1(\Sigma'', \mathcal{O}_{\Sigma''}(K_{\Sigma''})) = 0.$ 
\end{remark}

In our case, first we blow-up $\mathbb{P}^{3}$ at the vertices of $T$, at the point $p$ and at the six points $r_i\cap l_{0i}$, $r_i \cap l_{jk}$, for $i,j,k\in\{1,2,3\}$ with $j<k$ and $j,k\ne i$. In this way we obtain a smooth threefold $X'$ and a birational morphism $bl' : X' \to \mathbb{P}^3$ with exceptional divisors
$$E_h = bl'^{-1}(v_h), \quad E_p = bl'^{-1}(p), \quad E_i' = bl'^{-1}(r_i\cap l_{0i}), \quad E_i'' = bl'^{-1}(r_i\cap l_{jk}),$$
where $0\le h\le 3$.
Let us denote by $\widetilde{l}_{0i}$, $\widetilde{l}_{jk}$ and $\widetilde{r}_i$, respectively, the strict transforms of the lines $l_{0i}$, $l_{jk}$ and $r_i$. Then we blow-up $X'$ along these objects. We obtain a smooth threefold $X''$ and a birational morphism $bl'' : X'' \to X'$, with exceptional divisors 
$$F_{0i} = bl''^{-1}(\widetilde{l}_{0i}), \quad F_{jk} = bl''^{-1}(\widetilde{l}_{jk}), \quad R_i = bl''^{-1}(\widetilde{r}_i).$$
Furthermore, let us denote by $\widetilde{E}_h$, $\widetilde{E}_p$, $\widetilde{E}_i'$, $\widetilde{E}_i''$, respectively the strict transforms of $E_h$, $E_p$, $E_i'$, $E_i''$. 
We denote by $H$ the pullback of a general plane of $\mathbb{P}^{3}$ via the birational morphism $bl'\circ bl'' : X'' \to \mathbb{P}^{3}$.
Then the strict transform $\Sigma_{\bullet}''$ of an element $\Sigma_{\bullet} \in \mathcal{S}_{\bullet}$, via the blow-ups $bl'\circ bl'' : X'' \to \mathbb{P}^{3}$, is
$$ \Sigma_{\bullet}'' \sim 6H -3\widetilde{E}_p-\sum_{i=0}^3 3\widetilde{E}_i -\sum_{i=1}^{3}2\widetilde{E}_i' - \sum_{i=1}^{3}2\widetilde{E}_i'' - \sum_{0\le i<j\le 3}2F_{ij} - \sum_{i=1}^{3}R_i.$$

\begin{remark}\label{rem:KX}
The anticanonical divisor of $X''$ is linearly equivalent to the strict transform of a quartic surface of $\mathbb{P}^{3}$ with double points at the vertices of $T$ and at the point $p$ and containing the six edges of $T$ and the three lines $r_1,r_2,r_3$, i.e.
$$ K_{X''} \sim - 4H + 2\widetilde{E}_p+ 2\sum_{i=0}^3\widetilde{E}_i + \sum_{i=1}^{3}2\widetilde{E}_i' + \sum_{i=1}^{3}2\widetilde{E}_i'' + \sum_{0\le i<j\le 3}F_{ij} + \sum_{i=1}^{3}R_i$$
(see \cite[p.187]{GH}). Then we have $K_{X''}+ \Sigma_{\bullet}'' \sim  2H - \widetilde{E}_p -\sum_{i=0}^3\widetilde{E}_i-\sum_{0\le i<j\le 3}F_{ij}.$
Since there are no quadric surfaces of $\mathbb{P}^{3}$ containing the edges of $T$ (see Remark~\ref{rem:classicalcase}), there are also no quadric surfaces of $\mathbb{P}^{3}$ containing the edges of $T$ and the point $p$. So we obtain $h^0(X'', \mathcal{O}_{X''}(K_{X''}+\Sigma_{\bullet}''))=0$.
\end{remark}

\begin{theorem}\label{thm:propertiesSigma*}
Let $\mathcal{S}_{\bullet}$ be the linear system on $\mathbb{P}^{3}$ given by the sextic surfaces of $\mathbb{P}^{3}$ double along the six edges of the tetrahedron $T$ and triple at the general point $p \in \mathbb{P}^{3}$. The strict transform $\Sigma_{\bullet}''$ on $X''$ of a general element $\Sigma_{\bullet}\in \mathcal{S}_{\bullet}$, via the blow-ups described above, is a smooth surface with $p_{g}(\Sigma_{\bullet}'') =0$, $q(\Sigma_{\bullet}'') =1$ and $p_{a}(\Sigma_{\bullet}'') =-1$.
\end{theorem}
\begin{proof}
It is known that the blow-ups $bl'\circ bl'' : X'' \to \mathbb{P}^3$ solve the singularities of a general $\Sigma_{\bullet}\in \mathcal{S}_{\bullet}\subset \mathcal{S}$ at the vertices of the tetrahedron $T$ and along its edges. In order to obtain the smoothness of the strict transform $\Sigma_{\bullet}''$ on $X''$ of $\Sigma_{\bullet}$, it remains to show that $bl'\circ bl'' : X'' \to \mathbb{P}^3$ also solves the triple point $p$ of $\Sigma_{\bullet}$. By Bertini's Theorem, it is sufficient to prove that the linear system $|\Sigma_{\bullet}''|$ is base point free on $\widetilde{E}_p$. We recall that $\widetilde{E}_p$ is the blow-up of the plane $E_p\cong \mathbb{P}^2$ at the three points $E_p\cap \widetilde{r}_1$, $E_p\cap \widetilde{r}_2$, $E_p\cap \widetilde{r}_3$.
We also recall that $\Sigma_{\bullet}' \cap E_p= \mathbb{P}(TC_p\Sigma_{\bullet})$, where $\Sigma_{\bullet}':=bl''(\Sigma_{\bullet}'')$ and where $TC_p\Sigma_{\bullet}$ denotes the tangent cone to $\Sigma_{\bullet}$ at $p$. Thanks to a computational analysis via Macaulay2, we find that $\mathbb{P}(TC_p\Sigma_{\bullet})$ is a cubic plane curve passing through the points $E_p\cap \widetilde{r}_1$, $E_p\cap \widetilde{r}_2$, $E_p\cap \widetilde{r}_3$ (see Code~\ref{code:castelnuovoClassical} of Appendix~\ref{app:code}).
In particular we have that $|\Sigma_{\bullet}''|$ cuts on $\widetilde{E}_p$ the strict transform via $bl''|_{\widetilde{E}_p} : \widetilde{E}_p \to E_p$ of a linear system of cubic curves on $E_p$ whose base points are only the points $E_p\cap \widetilde{r}_1$, $E_p\cap \widetilde{r}_2$, $E_p\cap \widetilde{r}_3$ (see Code~\ref{code:castelnuovoClassical} of Appendix~\ref{app:code}). Thus, $|\Sigma_{\bullet}''| |_{\widetilde{E}_p}$ is base point free and so $\Sigma_{\bullet}''$ is smooth.
By using the adjunction formula we have the following exact sequence
$$0 \to \mathcal{O}_{X''}(K_{X''})  \to \mathcal{O}_{X''}(K_{X''} + \Sigma_{\bullet}'') \to \mathcal{O}_{\Sigma_{\bullet}''}(K_{\Sigma_{\bullet}''}) \to 0.$$
Since $X''$ is a smooth rational threefold, we have that $h^{i=0,1,2} (X'', \mathcal{O}_{X''}(K_{X''})) = 0$ by Serre Duality.
Then we obtain
$$p_g (\Sigma_{\bullet}'') = h^0 (\Sigma_{\bullet}'', \mathcal{O}_{\Sigma_{\bullet}''}(K_{\Sigma_{\bullet}''})) = h^0(X'', \mathcal{O}_{X''}(K_{X''}+\Sigma_{\bullet}''))= 0$$
(see Remark~\ref{rem:KX}). Furthermore, we have that
$$q (\Sigma_{\bullet}'') = h^1 (\Sigma_{\bullet}'', \mathcal{O}_{\Sigma_{\bullet}''}) = h^1 (\Sigma_{\bullet}'', \mathcal{O}_{\Sigma_{\bullet}''} (K_{\Sigma_{\bullet}''})) = h^1 (X'', \mathcal{O}_{X''}(K_{X''} + \Sigma_{\bullet}'')).$$ 
In order to verify that the last value is equal to $1$, we observe that the strict transform on $X''$ of a quadric surface of $\mathbb{P}^{3}$ containing the edges of $T$ is linearly equivalent to $2H - \sum_{i=0}^3 \widetilde{E}_i - \sum_{0\le i<j \le 3}F_{ij}$.
By Remark~\ref{rem:KX} we have the following exact sequence
$$0 \to \mathcal{O}_{X''}(K_{X''}+\Sigma_{\bullet}'')  \to \mathcal{O}_{X''}(2H - \sum_{i=0}^3 \widetilde{E}_i - \sum_{0\le i<j \le 3}F_{ij}) \to \mathcal{O}_{E_p} \to 0.$$
Since
$h^{i=0,1}(X'', \mathcal{O}_{X''}(2H - \sum_{i=0}^3 \widetilde{E}_i - \sum_{0\le i<j \le 3}F_{ij})) = 0$
(see Remark~\ref{rem:classicalcase}), then 
$h^1 (X'', \mathcal{O}_{X''}(K_{X''} + \Sigma_{\bullet}'')) =h^0(E_p, \mathcal{O}_{E_p})=h^0(\mathbb{P}^{2}, \mathcal{O}_{\mathbb{P}^{2}})=1.$
Finally, by the Riemann-Roch theorem we have that $p_{a}(\Sigma_{\bullet}'') = p_{g}(\Sigma_{\bullet}'')- q(\Sigma_{\bullet}'') = -1.$
\end{proof}

Let us now recall some definitions. Let $R$ be a smooth surface and $\Gamma$ a smooth, irreducible curve. We say that $R$ is a \textit{ruled surface over} $\Gamma$ if there is a surjective morphism $f : R \to \Gamma$ such that, for a general point $x\in \Gamma$, we have that $f^{-1}(x)$ is isomorphic to $\mathbb{P}^{1}$. 
It is equivalent to saying that $R$ is birational to $\Gamma \times \mathbb{P}^{1}$ (see \cite[Theorem III.4]{Beau78}). Furthermore, we say that a smooth variety $Z$ is \textit{uniruled} if it is covered by a family of rational curves. More precisely, $Z$ is a uniruled variety if there is a variety $K$ with $\dim K = \dim Z -1$ and there is a dominant rational map $ K \times \mathbb{P}^{1} \dashrightarrow Z$. Every uniruled variety $Z$ has Kodaira dimension $\kappa (Z) = -\infty$.

\begin{theorem}\label{thm:elliptic ruled surface}
Let $\mathcal{S}_{\bullet}$ be the linear system on $\mathbb{P}^{3}$ given by the sextic surfaces of $\mathbb{P}^{3}$ double along the six edges of the tetrahedron $T$ and triple at the general point $p \in \mathbb{P}^{3}$. 
Then the strict transform $\Sigma_{\bullet}''$ of a general element $\Sigma_{\bullet}\in \mathcal{S}_{\bullet}$, via the blow-ups described above, is an elliptic ruled surface.
\end{theorem}
\begin{proof}
Let us take a general $\Sigma_{\bullet}\in \mathcal{S}_{\bullet}$ and its image $\Gamma:=\nu_{\bullet}(\Sigma_{\bullet})$, which is a general hyperplane section of the Cayley cubic surface $\Delta\subset\mathbb{P}^{3}$. Since $\Delta$ only has isolated singularities, then $\Gamma$ is a smooth elliptic cubic plane curve. Furthermore, $\Sigma_{\bullet}$ is union of $\infty^{1}$ rational cubic plane curves, fibres of $\nu_{\bullet}$, given by the preimages of the $\infty^{1}$ points of $\Gamma$ (see Lemma~\ref{lem:contraction of congruence} and Remark~\ref{rem:generalfibreM2}). So $(\nu_{\bullet}\circ bl''\circ bl') : \Sigma_{\bullet}'' \to \Gamma$ is a uniruled surface. Since $\kappa (\Sigma_{\bullet}'') = -\infty$, we have that $\Sigma_{\bullet}''$ is an irrational elliptic ruled surface by Enriques-Kodaira classification and by Theorem~\ref{thm:propertiesSigma*}. 
\end{proof}

By construction, for a general surface $S_{\bullet}\in \mathcal{L}_{\bullet}$ there exists a unique surface $\Sigma_{\bullet}\in \mathcal{S}_{\bullet}$ such that $S_{\bullet} = \nu_{\bullet} (\Sigma_{\bullet})$. So if we denote by $\phi_{\bullet} : W \dashrightarrow \mathbb{P}^{3}$ the rational map defined by the linear system $\mathcal{L}_{\bullet}$, we have the following commutative diagram
$$\begin{tikzcd}
\mathbb{P}^3 \arrow[d, dashrightarrow, "\nu"] \arrow[dr, dashrightarrow, "\nu_{\bullet}"] & \\
W \arrow[d, hookrightarrow, "\phi_{\mathcal{L}}"]   \arrow[r, dashrightarrow, "\phi_{\bullet}"] & \Delta\subset\mathbb{P}^{3}\\
\mathbb{P}^{13} &
\end{tikzcd}$$
and we obtain Theorem~\ref{thm:MAINthmCONJ} (see Theorems~\ref{thm: image Cayley cubic},~\ref{thm:elliptic ruled surface}). We have thus proved that case (\ref{castelnuovoA}) of Castelnuovo's conjecture occurs for the classical Enriques-Fano threefold and that the consequences stated in Remark~\ref{rem:caseEllipticGamma} are verified.

\section{Consequences for the Prokhorov-Enriques-Fano threefolds}\label{sec: generalization}

It is known that all Enriques surfaces appear as the desingularization of some sextic surface of $\mathbb{P}^{3}$ double along the six edges of a tetrahedron and triple at the four vertices (see \cite[p.275]{CoDo89}).
By using notations of previous sections, we can say that all Enriques surfaces are birational to a surface $\Sigma\in \mathcal{S}$ and so to a hyperplane section of the classical Enriques-Fano threefold.
If we consider an Enriques-Fano threefold $(W, \mathcal{L})$ of genus $13 \le p \le 17$, we can say that a general $S\in \mathcal{L}$ is birational to a hyperplane section of the classical Enriques-Fano threefold. 
In particular, a general element of $\mathcal{L}$ having a triple point at a general point $w\in W$ is birational to a hyperplane section of the classical Enriques-Fano threefold 
with triple point at a point on it.
Let $(W_P^{17},\mathcal{L})$ be the Prokhorov-Enriques-Fano threefold of genus $17$: since a general $S\in \mathcal{L}$ is a general Enriques surface (see proof of \cite[Proposition 4.7]{CDGK20}), then, by Theorem~\ref{thm:MAINthmCONJ}, we obtain Corollary~\ref{cor:castelnuovop17}.

\subsection*{Open questions}
%{\color{blue}
Verifying whether the linear system $\mathcal{L}_{\bullet}$ on the Prokhorov-

-Enriques-Fano threefold of genus $17$ has dimension $7=17-10$ would be an interesting question to address in later studies.
Similarly, understanding what happens on the Prokhorov-Enriques-Fano threefold of genus $13$ would account for a stimulating topic of investigation.
%}

\appendix
\section{Computational analysis and Macaulay2 codes}\label{app:code}

We will essentially use the package \textit{Cremona} of Staglian\`{o} (see \cite{Sta18}) and in particular the following functions, commands and methods:
%\begin{itemize}
%\item 
\textit{toMap}, to construct the rational map defined by a linear system;
%\item 
\textit{rationalMap}, to construct rational maps between projective varieties;
%\item 
\textit{image}, to compute the image of a rational map;
%\item 
\textit{degree}, to compute the degree of a rational map;
%\item 
\textit{isBirational}, to verify the birationality of a rational map;
%\item 
\textit{inverseMap}, to compute the inverse of a birational map;
%\item 
\textit{ideal}, to compute the base locus of a rational map.
%\end{itemize}
We will also use the function \textit{tangentCone}, to compute the tangent cone to an affine variety at the origin, and the following standard functions: \textit{associatedPrimes}, to compute the irreducible components of a variety; \textit{jacobian}, to compute the Jacobian matrix of the generators of an ideal; \textit{minors}, to compute the ideal generated by the minors of a certain order of a given matrix.

\begin{code}\label{code:bayle13}
Let $W_{BS}^{13}$ be the Enriques-Fano threefold of genus $13$ found by Bayle and Sano (see \cite[\S 6.3.2]{Ba94} and \cite[Theorem 1.1 No.14]{Sa95}) and let us use the notation of \S~\ref{sec:EFGENUSgreater13}. Let us illustrate how to show, with a computational approach, that the embedding of $W_{BS}^{13}$ in $\mathbb{P}^{13}$ is the classical Enriques-Fano threefold. First we will explain the strategy to use and then we will give the Macaulay2 code.

We construct $W_{BS}^{13}$ by using the fact that the $\sigma$-invariant multihomogeneous polynomials of multidegree $(2,2,2)$ define the coordinates of the quotient morphism $\pi : X \to W_{BS}^{13} \subset \mathbb{P}^{13}_{\left[w_0:\dots :w_{13}\right]}$, that is
\begin{center}
$\left[x_0:x_1\right] \times \left[y_0:y_1\right] \times \left[z_0:z_1\right]$

$\downmapsto^\pi$

$[ x_0^2y_0^2z_0^2: x_0^2y_0^2z_1^2: x_0^2y_0y_1z_0z_1: x_0^2y_1^2z_0^2: x_0^2y_1^2z_1^2: x_0x_1y_0^2z_0z_1: x_0x_1y_0y_1z_0^2:$

$x_0x_1y_0y_1z_1^2:  x_0x_1y_1^2z_0z_1: x_1^2y_0^2z_0^2: x_1^2y_0^2z_1^2: x_1^2y_0y_1z_0z_1: x_1^2y_1^2z_0^2: x_1^2y_1^2z_1^2 ].$
\end{center}
We find that the threefold $W_{BS}^{13}\subset \mathbb{P}^{13}$ has ideal generated by the following $42$ quadratic polynomials:
\begin{center}
${w}_{10} {w}_{12}-{w}_{9} {w}_{13},\,\,
{w}_{7}{w}_{12}-{w}_{6} {w}_{13},\,\,
{w}_{4} {w}_{12}-{w}_{3} {w}_{13},\,\,
{w}_{1}{w}_{12}-{w}_{0} {w}_{13},\,\,
{w}_{11}^{2}-{w}_{9} {w}_{13},$

${w}_{8}{w}_{11}-{w}_{6} {w}_{13},\,\,
{w}_{7} {w}_{11}-{w}_{5} {w}_{13},\,\,
{w}_{6}{w}_{11}-{w}_{5} {w}_{12},\,\,
{w}_{4} {w}_{11}-{w}_{2} {w}_{13},\,\,
{w}_{3}{w}_{11}-{w}_{2} {w}_{12},$

${w}_{2} {w}_{11}-{w}_{0} {w}_{13},\,\,
{w}_{8}{w}_{10}-{w}_{5} {w}_{13},\,\,
{w}_{6} {w}_{10}-{w}_{5} {w}_{11},\,\,
{w}_{4}{w}_{10}-{w}_{1} {w}_{13},\,\,
{w}_{3} {w}_{10}-{w}_{0} {w}_{13},$

${w}_{2}{w}_{10}-{w}_{1} {w}_{11},\,\,
{w}_{8} {w}_{9}-{w}_{5} {w}_{12},\,\,
{w}_{7}{w}_{9}-{w}_{5} {w}_{11},\,\,
{w}_{4} {w}_{9}-{w}_{0} {w}_{13},\,\,
{w}_{3}{w}_{9}-{w}_{0} {w}_{12},$

${w}_{2} {w}_{9}-{w}_{0} {w}_{11},\,\,
{w}_{1}{w}_{9}-{w}_{0} {w}_{10},\,\,
{w}_{8}^{2}-{w}_{3} {w}_{13},\,\,
{w}_{7}{w}_{8}-{w}_{2} {w}_{13},\,\,
{w}_{6} {w}_{8}-{w}_{2} {w}_{12},$

${w}_{5}{w}_{8}-{w}_{0} {w}_{13},\,\,
{w}_{7}^{2}-{w}_{1} {w}_{13},\,\,
{w}_{6}{w}_{7}-{w}_{0} {w}_{13},\,\,
{w}_{5} {w}_{7}-{w}_{1} {w}_{11},\,\,
{w}_{3}{w}_{7}-{w}_{2} {w}_{8},$

${w}_{2} {w}_{7}-{w}_{1}{w}_{8},\,\,
{w}_{6}^{2}-{w}_{0} {w}_{12},\,\,
{w}_{5} {w}_{6}-{w}_{0}{w}_{11},\,\,
{w}_{4} {w}_{6}-{w}_{2} {w}_{8},\,\,
{w}_{2} {w}_{6}-{w}_{0}{w}_{8},$

${w}_{1} {w}_{6}-{w}_{0} {w}_{7},\,\,
{w}_{5}^{2}-{w}_{0}{w}_{10},\,\,
{w}_{4} {w}_{5}-{w}_{1} {w}_{8},\,\,
{w}_{3} {w}_{5}-{w}_{0}{w}_{8},\,\,
{w}_{2} {w}_{5}-{w}_{0} {w}_{7},\,\,
{w}_{1} {w}_{3}-{w}_{0}{w}_{4},$

${w}_{2}^{2}-{w}_{0} {w}_{4}.$
\end{center}
We project $\mathbb{P}^{13}$ from the $\mathbb{P}^7$ spanned by the eight singular points of $W_{BS}^{13}$: we obtain the rational map  
$\rho: \mathbb{P}^{13} \dashrightarrow \mathbb{P}^{5}$, $\left[w_0 : \dots : w_{13} \right] \mapsto \left[w_2: w_5: w_6: w_7: w_8: w_{11} \right].$ 
In particular the restriction map $\rho|_{ W_{BS}^{13}} : W_{BS}^{13} \dashrightarrow \mathbb{P}^{5}$ is birational onto the image, which is the quartic threefold $Q:=\{t_1t_4  - t_0t_5 = 0,\, t_2t_3  - t_0t_5 = 0\}$ given by the complete intersection of two quadric hypersurfaces of $\mathbb{P}^{5}_{\left[t_0:\dots :t_5\right]}$.
Such a threefold $Q$ is birational to $\mathbb{P}^{3}_{\left[s_0:\dots :s_3\right]}$ via the rational map defined by the linear system of the quadric surfaces passing through the four vertices of the tetrahedron $\{s_0s_1s_2s_3=0\}$, i.e.
$q : \mathbb{P}^{3} \dashrightarrow Q \subset \mathbb{P}^{5}, \quad
\left[s_0 : s_1 : s_2 : s_3\right] \mapsto \left[s_0 s_1 : s_1 s_2 : s_1 s_3 : s_0 s_2 : s_0 s_3 : s_2 s_3\right].$
We can compute the inverse map of $q$, which is the map
$q^{-1} : Q \subset \mathbb{P}^{5} \dashrightarrow\mathbb{P}^{3}$ given by 
$\left[t_0 : t_1 : t_2 : t_3 : t_4 : t_5\right] \mapsto \left[t_3t_4 : t_0 t_5 : t_3t_5 : t_4 t_5\right].$
Thus, we have the birational map
$(q^{-1}\circ \rho |_{W_{BS}^{13}}) : W_{BS}^{13} \subset \mathbb{P}^{13} \dashrightarrow \mathbb{P}^{3}$,  $\left[w_0 : \dots : w_{13}\right] \mapsto \left[w_7 w_8 : w_2 w_{11} : w_7 w_{11} : w_8 w_{11}\right].$ We compute its inverse map $\nu : \mathbb{P}^{3} \dashrightarrow  W_{BS}^{13} \subset \mathbb{P}^{13}$.
We obtain that $\nu$ is the map defined by
$\left[s_0 : s_1 : s_2 : s_3 \right] \mapsto  \left[w_0 : \dots : w_{13}\right],$
where
$w_{0} = {s}_{0} {s}_{1}^{3} {s}_{2} {s}_{3}$,
$w_{1} = {s}_{0}^{2} {s}_{1}^{2} {s}_{2}^{2}$,
$w_{2} = {s}_{0}^{2} {s}_{1}^{2} {s}_{2} {s}_{3}$,
$w_{3} = {s}_{0}^{2} {s}_{1}^{2} {s}_{3}^{2}$,
$w_{4} = {s}_{0}^{3} {s}_{1} {s}_{2} {s}_{3},$
$w_{5} = {s}_{0} {s}_{1}^{2} {s}_{2}^{2} {s}_{3}$,
$w_{6} = {s}_{0} {s}_{1}^{2} {s}_{2} {s}_{3}^{2}$,
$w_{7} = {s}_{0}^{2} {s}_{1} {s}_{2}^{2} {s}_{3}$,
$w_{8} = {s}_{0}^{2} {s}_{1} {s}_{2} {s}_{3}^{2}$,
$w_{9} = {s}_{1}^{2} {s}_{2}^{2} {s}_{3}^{2},$
$w_{10} = {s}_{0} {s}_{1} {s}_{2}^{3} {s}_{3}$,
$w_{11} = {s}_{0} {s}_{1} {s}_{2}^{2} {s}_{3}^{2}$,
$w_{12} = {s}_{0} {s}_{1} {s}_{2} {s}_{3}^{3}$,
$w_{13} = {s}_{0}^{2} {s}_{2}^{2} {s}_{3}^{2}.$
Thus, $\nu$ is the rational map defined by the linear system of the sextic surfaces of $\mathbb{P}^{3}$ double along the six edges of the tetrahedron $\{s_0s_1s_2s_3=0\}$ and so $W_{BS}^{13}\subset \mathbb{P}^{13}$ is the classical Enriques-Fano threefold.

\medskip

\begin{scriptsize}
\begin{verbatim}
Macaulay2, version 1.11
with packages: ConwayPolynomials, Elimination, IntegralClosure, InverseSystems,
               LLLBases, PrimaryDecomposition, ReesAlgebra, TangentCone
i1 : needsPackage "Cremona";
i2 : PP1x = ZZ/10000019[x_0,x_1];
i3 : PP1y = ZZ/10000019[y_0,y_1];
i4 : PP1z = ZZ/10000019[z_0,z_1];
i5 : X = PP1x ** PP1y ** PP1z;
i6 : use X;
i7 : pigreca = rationalMap map(X, ZZ/10000019[w_0..w_13], matrix{{x_0^2*y_0^2*z_0^2, 
     x_0^2*y_0^2*z_1^2, x_0^2*y_0*y_1*z_0*z_1, x_0^2*y_1^2*z_0^2, x_0^2*y_1^2*z_1^2, 
     x_0*x_1*y_0^2*z_0*z_1, x_0*x_1*y_0*y_1*z_0^2, x_0*x_1*y_0*y_1*z_1^2, 
     x_0*x_1*y_1^2*z_0*z_1, x_1^2*y_0^2*z_0^2, x_1^2*y_0^2*z_1^2, x_1^2*y_0*y_1*z_0*z_1, 
     x_1^2*y_1^2*z_0^2, x_1^2*y_1^2*z_1^2}});
i8 : WB13 = image pigreca
i9 : (dim WB13 -1, degree WB13) == (3, 24)
i10 : PP13 = ring WB13;
i11 : P1 = pigreca(ideal{x_1,y_0,z_0});
i12 : P2 = pigreca(ideal{x_1,y_1,z_1});
i13 : P3 = pigreca(ideal{x_0,y_1,z_0});
i14 : P4 = pigreca(ideal{x_0,y_0,z_1});
i15 : P1' = pigreca(ideal{x_0,y_1,z_1});
i16 : P2' = pigreca(ideal{x_0,y_0,z_0});
i17 : P3' = pigreca(ideal{x_1,y_0,z_1});
i18 : P4' = pigreca(ideal{x_1,y_1,z_0});
i19 : proj1 = rationalMap toMap(P1,1,1);
i20 : proj2 = rationalMap toMap(proj1(P2),1,1);
i21 : proj3 = rationalMap toMap(proj2(proj1(P3)),1,1);
i22 : proj4 = rationalMap toMap(proj3(proj2(proj1(P4))),1,1);
i23 : proj5 = rationalMap toMap(proj4(proj3(proj2(proj1(P1')))),1,1);
i24 : proj6 = rationalMap toMap(proj5(proj4(proj3(proj2(proj1(P2'))))),1,1);
i25 : proj7 = rationalMap toMap(proj6(proj5(proj4(proj3(proj2(proj1(P3')))))),1,1);
i26 : proj8 = rationalMap toMap(proj7(proj6(proj5(proj4(proj3(proj2(proj1(P4'))))))),1,1);
i27 : rho = proj1*proj2*proj3*proj4*proj5*proj6*proj7*proj8
i28 : Q = rho(WB13)
i29 : (dim Q -1, degree Q) == (3, 4)
i30 : isBirational((rho|WB13)||Q) == true
i31 : PP5 = ring Q;
i32 : PP3 = ZZ/10000019[t_0..t_3];
i33 : q = rationalMap map(PP3, PP5, matrix{{(gens PP3)_0*(gens PP3)_1, 
      (gens PP3)_1*(gens PP3)_2,(gens PP3)_1*(gens PP3)_3, (gens PP3)_0*(gens PP3)_2, 
      (gens PP3)_0*(gens PP3)_3, (gens PP3)_2*(gens PP3)_3}});
i34 : (image q == Q) == true
i35 : isBirational(q||Q)
i36 : mapP5toP3 = rationalMap map(PP5, PP3, sub(matrix(inverseMap(q||Q)), PP5)) 
i37 : mapWB13toP3 = (rho*mapP5toP3) | WB13;
i38 : (isBirational mapWB13toP3) == true
i39 : nu = rationalMap map( PP3, ring WB13, matrix(inverseMap(mapWB13toP3))) 
i40 : (image nu == WB13) == true
\end{verbatim}
\end{scriptsize}
\end{code}

\begin{code}\label{code:castelnuovoClassical}
Let $\mathcal{S}_{\bullet}$ be the linear system on $\mathbb{P}^{3}$ given by the sextic surfaces of $\mathbb{P}^{3}$ double along the six edges of a tetrahedron $T$ and triple at a general point $p \in \mathbb{P}^{3}$. Let us use the notation of \S~\ref{sec: classical EF3-fold} and in particular let us see the proof of Theorem~\ref{thm:propertiesSigma*}. Let $\Sigma_{\bullet}$ be a general element of $\mathcal{S}_{\bullet}$ and let $\pi$ be a general plane of $\mathbb{P}^3$, that is a plane not containing the point $p$. Thanks to Macaulay2, one can find that the tangent cone to $\Sigma_{\bullet}$ at $p$ is a cone with vertex $p$ over a cubic plane curve on $\pi$ passing through the three points $\pi\cap r_1$, $\pi\cap r_3$ and $\pi\cap r_3$. In particular, by moving the surface $\Sigma_{\bullet}\in \mathcal{S}_{\bullet}$, these cubic cones cut on $\pi$ a linear system of cubic curves whose base locus is given exactly by the three points $\pi\cap r_1$, $\pi\cap r_3$ and $\pi\cap r_3$. Before providing the Macaulay2 code, let us explain the strategy to use: 
\begin{itemize}
\item[(i)] we consider the linear system $\mathcal{S}$ of the sextic surfaces of $\mathbb{P}^3_{\left[s_0:\dots :s_3\right]}$ having double points along the six edges of the tetrahedron $T:=\{s_0s_1s_2s_3=0\}$, which has equation
$$l_0s_0s_1^3s_2s_3+l_1s_0^2s_1^2s_2^2+l_2s_0^2s_1^2s_2s_3+
     l_3s_0^2s_1^2s_3^2+l_4s_0^3s_1s_2s_3+$$
     $$+l_5s_0s_1^2s_2^2s_3
     +l_6s_0s_1^2s_2s_3^2+l_7s_0^2s_1s_2^2s_3+l_8s_0^2s_1s_2s_3^2+$$
     $$+l_9s_1^2s_2^2s_3^2+l_{10}s_0s_1s_2^3s_3+
     l_{11}s_0s_1s_2^2s_3^2+l_{12}s_0s_1s_2s_3^3+l_{13}s_0^2s_2^2s_3^2=0;$$
\item[(ii)] we choose a point $p\in \mathbb{P}^3$ sufficiently general so that, by setting it as a triple point for the surfaces of $\mathcal{S}$, it imposes 10 linearly independent conditions to the coefficients $l_0, \dots ,l_{13}$: in our example we choose $p:=\left[1:1:1:-1\right]$;
\item[(iii)] we find the equation of $\mathcal{S}_{\bullet}$: in our example we have
$$l_{10}(s_0^3s_1s_2s_3-2s_0^2s_1s_2^2s_3+s_0s_1s_2^3s_3+s_0^2s_1^2s_3^2-2s_0s_1^2s_2s_3^2+s_1^2s_2^2s_3^2)+$$
$$+l_{11}(s_0^2s_1^2s_2s_3-s_0^3s_1s_2s_3+s_0^2s_1s_2^2s_3-s_0s_1^2s_2^2s_3-s_0^2s_1s_2s_3^2+s_0s_1^2s_2s_3^2+$$
$$+s_0s_1s_2^2s_3^2-s_1^2s_2^2s_3^2)+l_{12}(s_0^2s_1^2s_2^2+s_0^3s_1s_2s_3+2s_0s_1^2s_2^2s_3+2s_0^2s_1s_2s_3^2+$$
$$+s_1^2s_2^2s_3^2+s_0s_1s_2s_3^3)+l_{13}(s_0s_1^3s_2s_3+2s_0^2s_1s_2^2s_3-s_0^3s_1s_2s_3-2s_0s_1^2s_2^2s_3+$$
$$-2s_0^2s_1s_2s_3^2+2s_0s_1^2s_2s_3^2+s_0^2s_2^2s_3^2-s_1^2s_2^2s_3^2)=0.$$
We see that a general fibre of the rational map defined by $\mathcal{S}_{\bullet}$ is a cubic plane curve with node at $p$ and intersecting each edge of $T$ at a point.
We also recall that the base locus of $\mathcal{S}_{\bullet}$ is given by the union of the six edges of $T$ and by three lines $r_1$, $r_2$, $r_3$ intersecting at $p$ (see Corollary~\ref{cor:onlybasecurves edgesT and ri});
\item[(iv)] we consider a change of coordinates of $\mathbb{P}^{3}$, with respect to which $p$ has coordinates $\left[0: 0: 0: 1\right]$. By abuse of notation let us denote the new coordinates by $\left[s_0:\dots :s_3\right]$. Let $\Sigma_{\bullet}$ be a general element of $\mathcal{S}_{\bullet}$, obtained by fixing general values for $l_{10}, \dots, l_{13}$. The point $p$ can be viewed as the origin of the open affine set $U_0 := \{s_3 \ne 0\}$ and we can find the ideal of the tangent cone $TC_{p}(\Sigma_{\bullet}\cap U_0)$: in our example we obtain
$$(l_{10}-l_{11}+l_{12}-l_{13})s_0^3+(-l_{10}+l_{11}-l_{12}+l_{13})s_0^2s_1-l_{13}s_0s_1^2+l_{13}s_1^3+$$
$$-(l_{10}-l_{11}+l_{12}-l_{13})s_0^2s_2+(2l_{10}-l_{11})s_0s_1s_2-l_{13}s_1^2s_2-l_{10}s_0s_2^2+$$
$$-l_{10}s_1s_2^2+l_{10}s_2^3=0,$$
thus, $TC_p\Sigma_{\bullet}$ is a cone with vertex $p$ over a cubic plane curve on the plane $\pi:=\{s_3=0\}$; 
\item[(v)] by moving $\Sigma_{\bullet}\in \mathcal{S}_{\bullet}$, i.e. by varying the coefficients $l_{10}, \dots , l_{13}$, the cubic cones $TC_p\Sigma_{\bullet}$ identify a linear system of cubic plane curves on $\pi$; we see that the base locus of this linear system is given by the union of the three points $r_1\cap \pi$, $r_2\cap \pi$, $r_3\cap \pi$: we verify this by studying the intersection of the four cubic curves given by $\left[l_{10}:\dots :l_{13}\right] \in\{\left[1:0:0:0\right]$, $\left[0:1:0:0\right]$, $\left[0:0:1:0\right]$, $\left[0:0:0:1\right]\}$.
\end{itemize}

\medskip

\begin{scriptsize}
\begin{verbatim}
Macaulay2, version 1.11
with packages: ConwayPolynomials, Elimination, IntegralClosure, InverseSystems,
               LLLBases, PrimaryDecomposition, ReesAlgebra, TangentCone
i1 : needsPackage "Cremona";
i2 : PP3 = ZZ/10000019[s_0..s_3];
i3 : -- let us take a general point of PP3 with random coordinates:
     -- for i to 3 list random(-5,10)
     -- in our example we take p=[1: 1: 1: -1]
     p = ideal{s_0+s_3,s_1+s_3,s_2+s_3}
i4 : -- let us take the linear system of the sextic surfaces of PP3 
     -- double along the six edges of the coordinate tetrahedron
     R = ZZ/10000019[l_0..l_13][s_0..s_3];
i5 : use R
i6 : Sigma = ideal{l_0*s_0*s_1^3*s_2*s_3+l_1*s_0^2*s_1^2*s_2^2+l_2*s_0^2*s_1^2*s_2*s_3+
     l_3*s_0^2*s_1^2*s_3^2+l_4*s_0^3*s_1*s_2*s_3+l_5*s_0*s_1^2*s_2^2*s_3+l_6*s_0*s_1^2*s_2*s_3^2+
     l_7*s_0^2*s_1*s_2^2*s_3+l_8*s_0^2*s_1*s_2*s_3^2+l_9*s_1^2*s_2^2*s_3^2+l_10*s_0*s_1*s_2^3*s_3
     +l_11*s_0*s_1*s_2^2*s_3^2+l_12*s_0*s_1*s_2*s_3^3+l_13*s_0^2*s_2^2*s_3^2};
i7 : -- for a fixed value of [l_0:..:l_13], we have that Sigma is a hypersurface of PP3
     -- let us find the values for [l_0:..:l_13] in order to have p as triple point for Sigma
     J = jacobian(Sigma);
i8 : JJ = jacobian(J);
i9 : triplelocus = minors(1,J)+minors(1,JJ)+Sigma;
i10 : substitute(triplelocus, {s_0=>1, s_1=>1, s_2=>1, s_3=>-1})
i11 : -- we have the following 10 independent conditions
      substitute(oo,{l_0 => l_13})
i12 : substitute(oo,{l_1 => l_12})
i13 : substitute(oo,{l_2 => l_11})
i14 : substitute(oo,{l_3 => l_10})
i15 : substitute(oo,{l_4 => l_10-l_11+l_12-l_13})
i16 : substitute(oo,{l_5 => -l_11 + 2*l_12 - 2*l_13})
i17 : substitute(oo,{l_6 => -2*l_10+l_11+2*l_13})
i18 : substitute(oo,{l_7 => -2*l_10+l_11+2*l_13})
i19 : substitute(oo,{l_8 => -l_11+2*l_12-2*l_13})
i20 : substitute(oo,{l_9 => l_10-l_11+l_12-l_13})
i21 : -- thus, we let:
      substitute(Sigma,{l_0 => l_13})
i22 : substitute(oo,{l_1 => l_12})
i23 : substitute(oo,{l_2 => l_11})
i24 : substitute(oo,{l_3 => l_10})
i25 : substitute(oo,{l_4 => l_10-l_11+l_12-l_13})
i26 : substitute(oo,{l_5 => -l_11 + 2*l_12 - 2*l_13})
i27 : substitute(oo,{l_6 => -2*l_10+l_11+2*l_13})
i28 : substitute(oo,{l_7 => -2*l_10+l_11+2*l_13})
i29 : substitute(oo,{l_8 => -l_11+2*l_12-2*l_13})
i30 : substitute(oo,{l_9 => l_10-l_11+l_12-l_13})
i31 : -- the linear system of the sextic surfaces of PP3
      -- double along the edges of the coordinate tetrahedron
      -- and triple at the point p has the following equation,
      -- depending on the coefficients l_10,l_11,l_12,l_13
      SigmaTripleAtp = oo
i32 : -- let us find the rational map defined by SigmaTripleAtp
      generator1 = substitute( SigmaTripleAtp, {l_10 =>1, l_11=>0, l_12=>0, l_13=>0})
i33 : generator2 = substitute( SigmaTripleAtp, {l_10 =>0, l_11=>1, l_12=>0, l_13=>0})
i34 : generator3 = substitute( SigmaTripleAtp, {l_10 =>0, l_11=>0, l_12=>1, l_13=>0})
i35 : generator4 = substitute( SigmaTripleAtp, {l_10 =>0, l_11=>0, l_12=>0, l_13=>1})
i36 : PP3' = ZZ/10000019[x_0..x_3]
i37 : sexticsbullet = rationalMap map(PP3,PP3',matrix{{sub(generator1_0,PP3),
      sub(generator2_0,PP3),sub(generator3_0,PP3),sub(generator4_0,PP3)}});
i38 : CayleyCubic = image oo
i39 : (dim oo -1, degree oo) == (2, 3)
i40 : -- let us find the general fibre of sexticsbullet
      gamma = sexticsbullet^*(sexticsbullet(ideal{random(1,PP3),random(1,PP3),random(1,PP3)}))
i41 : (dim oo -1, degree oo) == (1, 3)
i42 : alpha = ideal{gamma_0}
i43 : (dim oo -1, degree oo) == (2, 1)
i44 : (dim(gamma+ideal{(gens PP3)_0,(gens PP3)_1})-1, 
      degree(gamma+ideal{(gens PP3)_0,(gens PP3)_1}))==(0, 1)
i45 : (dim(gamma+ideal{(gens PP3)_0,(gens PP3)_2})-1, 
      degree(gamma+ideal{(gens PP3)_0,(gens PP3)_2}))==(0, 1)
i46 : (dim(gamma+ideal{(gens PP3)_0,(gens PP3)_3})-1, 
      degree(gamma+ideal{(gens PP3)_0,(gens PP3)_3}))==(0, 1)
i47 : (dim(gamma+ideal{(gens PP3)_1,(gens PP3)_2})-1, 
      degree(gamma+ideal{(gens PP3)_1,(gens PP3)_2}))==(0, 1)
i48 : (dim(gamma+ideal{(gens PP3)_1,(gens PP3)_3})-1, 
      degree(gamma+ideal{(gens PP3)_1,(gens PP3)_3}))==(0, 1)
i49 : (dim(gamma+ideal{(gens PP3)_2,(gens PP3)_3})-1, 
      degree(gamma+ideal{(gens PP3)_2,(gens PP3)_3}))==(0, 1)
i50 : (alpha+p == p, gamma+p == p) == (true, true)
i51 : (p == saturate(gamma+minors(2,jacobian(gamma)))) == true
i52 : -- let us find the base locus of SigmaTripleAtp
      associatedPrimes(ideal sexticsbullet)
i53 : -- it is the union of the six edges of T, the point p
      -- and the following three lines r1, r2, r3 intersecting at p
      -- such that ri intersects the edges ideal{s_j,s_k}, ideal{s_0,s_i}
      -- for i,j,k distinct indices in {1,2,3} 
      use PP3
i54 : r1 = ideal{s_2+s_3,s_0-s_1}
i55 : r2 = ideal{s_1+s_3,s_0-s_2}
i56 : r3 = ideal{s_1-s_2,s_0+s_3}
i57 : -- let us find the tangent cone at the point p
      -- to a general sextic surface of the linear system SigmaTripleAtp
      newR = ZZ/10000019[l_10,l_11,l_12,l_13][s_0..s_3];
i58 : -- let us consider the change of coordinates thanks to which
      -- the point p is the point [0:0:0:1]
      -- (by abuse of notation ,let [s_0..s_3] be the new coordinates)
      substitute(SigmaTripleAtp, newR)
i59 : sub(oo, {(gens newR)_0 =>(gens newR)_0-(gens newR)_3, 
      (gens newR)_1=>(gens newR)_1-(gens newR)_3, (gens newR)_2=>(gens newR)_2-(gens newR)_3, 
      (gens newR)_3=>(gens newR)_3});
i60 : sub(oo, {(gens newR)_3 => 1})
i61 : TCp = tangentCone oo
i62 : -- TCp is a cone of vertex p over a cubic plane curve on the plane ideal{s_3}.
      -- By moving the surfaces of the linear system, i.e. by varying the values 
      -- l_10,l_11,l_12,l_13, we obtain a linear system of cubic plane curves on 
      -- ideal{s_3} which only has three base points,
      -- given by the intersection with the three lines r1, r2, r3
      c0 =sub(ideal{sub(TCp,{l_10=>1, l_11=>0, l_12=>0, l_13=>0})},PP3)
i63 : c1 =sub(ideal{sub(TCp,{l_10=>0, l_11=>1, l_12=>0, l_13=>0})},PP3)
i64 : c2 =sub(ideal{sub(TCp,{l_10=>0, l_11=>0, l_12=>1, l_13=>0})},PP3)
i65 : c3 =sub(ideal{sub(TCp,{l_10=>0, l_11=>0, l_12=>0, l_13=>1})},PP3)
i66 : threepts = associatedPrimes(ideal{(gens PP3)_3}+c0+c1+c2+c3)
i67 : threepts#0 == ideal{(gens PP3)_3}+sub(r1, {(gens PP3)_0 =>(gens PP3)_0-(gens PP3)_3,
                    (gens PP3)_1=>(gens PP3)_1-(gens PP3)_3, 
                    (gens PP3)_2=>(gens PP3)_2-(gens PP3)_3, (gens PP3)_3=>(gens PP3)_3})
i68 : threepts#1 == ideal{(gens PP3)_3}+sub(r2, {(gens PP3)_0 =>(gens PP3)_0-(gens PP3)_3, 
                   (gens PP3)_1=>(gens PP3)_1-(gens PP3)_3, 
                   (gens PP3)_2=>(gens PP3)_2-(gens PP3)_3, (gens PP3)_3=>(gens PP3)_3})
i69 : threepts#2 == ideal{(gens PP3)_3}+sub(r3, {(gens PP3)_0 =>(gens PP3)_0-(gens PP3)_3, 
                    (gens PP3)_1=>(gens PP3)_1-(gens PP3)_3, 
                    (gens PP3)_2=>(gens PP3)_2-(gens PP3)_3, (gens PP3)_3=>(gens PP3)_3})
\end{verbatim}
\end{scriptsize}
\end{code}

\end{document}